\documentclass[a4paper,12pt]{article}
\usepackage{graphicx}
\usepackage{color}
\usepackage{epsfig}

\newcommand{\fhi}{\varphi}

\newcommand{\spann}{\mathrm{span}}

\newcommand{\eeta}{\eps}
\newcommand{\mt}{\mapsto}
\newcommand{\al}{\alpha}

\newcommand{\Sch}{Schr\"odinger}
\newcommand{\diag}{discrete-spectrum}

\newcommand{\ad}{\mathrm{ad}}

\newcommand{\beq}{\begin{equation}} 
\newcommand{\eeq}{\end{equation}}
\newcommand{\be}{\begin{equation}} 
\newcommand{\ee}{\end{equation}}
\newcommand{\bea}{\begin{eqnarray}} 
\newcommand{\eea}{\end{eqnarray}}
\newcommand{\brs}{\begin{eqnarray*}}
\newcommand{\ers}{\end{eqnarray*}}
\newcommand{\ba}{\begin{array}} 
\newcommand{\ea}{\end{array}}
\newcommand{\br}{\begin{eqnarray}}
\newcommand{\er}{\end{eqnarray}}

\newcommand{\Lie}{\mathrm{Lie}}
\newcommand{\lp}{\left(}
\newcommand{\rp}{\right)}
\newcommand{\la}{\left\langle}
\newcommand{\ra}{\right\rangle}

\def\EOP{\ \hfill $\Box$}
\def\EOOP{\ \hfill $\bullet$ }

\renewcommand{\r}[1]{(\ref{#1})}
\def\eps{\varepsilon}
\def\lb{\lambda}
\newcommand{\s}{{\cal S}}
\newcommand{\PPi}{{\overline \Pi}}

\usepackage{amsthm,amssymb,amsbsy,amsmath,amsfonts,amssymb,amscd,mathrsfs}

\theoremstyle{plain}
\newtheorem{theorem}{Theorem}[section]
\newtheorem{corol}[theorem]{Corollary}
\newtheorem{lem}[theorem]{Lemma}
\newtheorem{claim}[theorem]{Claim}
\newtheorem{prop}[theorem]{Proposition}
\theoremstyle{definition}
\newtheorem{defn}[theorem]{Definition}

\theoremstyle{remark}
\newtheorem{rem}[theorem]{Remark}
\numberwithin{equation}{section}

\renewcommand{\H}{\mathscr{H}}

\newcommand{\N}{{\mathbf N}}
\newcommand{\C}{{\mathbf{C}}}
\newcommand{\R}{{\mathbf{R}}}
\newcommand{\Q}{{\mathbf{Q}}}
\newcommand{\Z}{{\mathbf Z}}
\newcommand{\F}{{\mathbf F}}

\begin{document}

\title{Controllability of the discrete-spectrum  \Sch\ equation driven
by an external field}
\date{\today}
\maketitle

\medskip
\centerline{\scshape Thomas Chambrion}
\medskip
{\footnotesize
 \centerline{Institut \'Elie Cartan, UMR 7502 Nancy-Universit\'e/CNRS/INRIA, BP 239, 54506 Vand\oe uvre-l\`es-Nancy,
France}
\centerline{{\tt Thomas.Chambrion@iecn.u-nancy.fr}}}

\medskip
\centerline{\scshape Paolo Mason}
\medskip
{\footnotesize

\centerline{IAC, CNR, Viale Del Policlinico, 137
00161 Rome, Italy  and}
 \centerline{Institut \'Elie Cartan, UMR 7502 INRIA/Nancy-Universit\'e/CNRS, BP 239, 54506 Vand\oe uvre-l\`es-Nancy,
France}\centerline{{\tt p.mason@iac.cnr.it}}
}

\medskip
\centerline{\scshape Mario Sigalotti}
\medskip
{\footnotesize
 \centerline{Institut \'Elie Cartan, UMR 7502 INRIA/Nancy-Universit\'e/CNRS, BP 239, 54506 Vand\oe uvre-l\`es-Nancy,
France}\centerline{{\tt mario.sigalotti@inria.fr}}}

\medskip
\centerline{\scshape Ugo Boscain\footnote{The last author was partially supported by a FABER grant by {\it Conseil r\'egional de Bourgogne}}}
\medskip
{\footnotesize
\centerline{Le2i, CNRS, 
Universit\'e de Bourgogne, BP 47870, 21078
Dijon Cedex, France}
\centerline{{\tt ugo.boscain@u-bourgogne.fr}}}

\begin{quotation}
{\bf Abstract} 
We prove approximate controllability of the bilinear Schr\"odinger equation in the case in which the uncontrolled Hamiltonian has discrete  non-resonant spectrum. The results that are obtained apply both to bounded or unbounded domains and to the case in which the control potential is bounded or   
unbounded. The method relies on finite-dimensional techniques applied to the Galerkin approximations and permits, in addition,  to get some controllability properties for the density matrix. Two examples are presented: the harmonic oscillator and the 3D well of potential, both controlled by suitable potentials.
\end{quotation}

\medskip 

\begin{quotation}
{\bf R\'esum\'e} 
Nous montrons la contr\^olabilit\'e approch\'ee de  l'\'equation de Schr\"odinger bilin\'eaire dans le cas o\`u l'hamiltonien non contr\^ol\'e a un spectre discret et non-r\'esonnant. Les r\'esultats obtenus sont valables que le domaine soit born\'e ou non, et que le potentiel de contr\^ole soit born\'e ou non. La preuve repose sur des m\'ethodes de dimension finie appliqu\'ees aux approximations de Galerkyn du syst\`eme. Ces m\'ethodes permettent en plus d'obtenir des r\'esultats de contr\^olabilit\'e des matrices de densit\'e. Deux exemples sont pr\'esent\'es, l'oscillateur harmonique et le puit de potentiel en dimension trois, munis de potentiels de contr\^ole ad\'equats.  
\end{quotation}

\section{Introduction}

In this paper we study the controllability of the bilinear Schr\"odinger equation. Its importance is due to applications to modern technologies such as  Nuclear Magnetic Resonance, laser spectroscopy, and quantum information science (see for instance \cite{glaser,peirs,science,shapiro}).

Many controllability results are available when the state space is finite dimensional, e.g.,  for spin systems
or for molecular dynamics  when one neglects interactions with highly excited levels (see for instance \cite{altafini1,dalessandro-book}). When the state space is infinite-dimensional the controllability problem appears to be  much more intricate.
Some results are available when the control is the value of the wave function on some portion of the boundary or in 
some internal region of the domain  
(see \cite{Review_zuazua} and references therein and the recent paper \cite{TeTuc}).   

However, from the point of view of applications the case in which  the control appears in the Hamiltonian as an external field  is much more interesting, since  the wave function is not directly accessible in experiments  and because of the postulate of collapse of the wave function. For instance, in nuclear magnetic resonance  the control is a magnetic field, in laser spectroscopy and in many applications of photochemistry the control is a laser or a source of light.

In this paper we consider the controllability problem  for the following bilinear system representing
the Schr\"odinger equation  driven by one external field 
\begin{eqnarray}
i\frac{d\psi}{dt}(t)=(H_0+u(t) H_1)\psi(t).
\label{eqeq}
\end{eqnarray}
Here the wave function $\psi$ evolves in an infinite-dimensional Hilbert space, $H_0$  is a self-adjoint operator called  {\it drift Hamiltonian} (i.e. the Hamiltonian responsible for  the evolution when the external field is not active),   $u(t)$ is a scalar control function, and $H_1$ is a self-adjoint operator describing the interrelation between the system and the external field.   

The reference case is the one in which the Hilbert space is $L^2(\Omega)$  where $\Omega$ is either $\R^d$ or a  bounded domain of $\R^d$, 
and equation (\ref{eqeq}) reads
\begin{eqnarray}
i \frac{\partial \psi}{\partial t}(t,x)=\left( -\Delta  +V(x)+ u(t) W(x)\right)\psi(t,x),
\label{eqeq2}
\end{eqnarray}
where $\Delta$ is the Laplacian (with Dirichlet boundary condition in the case in which $\Omega$ is bounded) and $V$ and $W$ are  suitably regular functions defined on $\Omega$.
However the setting of the paper  covers more general cases (for instance 
 $\Omega$ can be a Riemannian manifold and $\Delta$  the corresponding Laplace-Beltrami operator). Let us stress that the proposed approach allows to handle both cases
where the control potential (i.e. $H_1$ in \r{eqeq} or $W$ in \r{eqeq2}) is
bounded or  unbounded. Notice that in many situations the control
potential happens to be unbounded. For instance
if $\Omega=\R^d$ and
 the  controlled 
 external force depends on time, but is constant in space, then
$W$ is linear and hence unbounded.

Besides the fact that one cannot expect exact controllability on the whole Hilbert sphere 
(see \cite{bms, turinici}) and some negative result (in particular \cite{mira_rouch,roucho2})
only few approximate controllability results are available 
and concern mainly special situations.
It should be mentioned, however, that several results 
on efficient steering of the Schr\"odinger equation without any 
controllability assumptions are available, e.g.~\cite{bkp,borzi,ito}. 
(For optimal control results for finite dimensional 
quantum systems see, for instance, \cite{boscain4,boscain3,q5,brockett}.)

In \cite{Beauchard1,beauchard-coron}
   Beauchard and Coron study  the controllability of a quantum particle in a 1D potential well with $W(x)=x$. 
Their results are highly nontrivial and are based on Coron's return method (see \cite{coron-libro}) and  Nash--Moser's theorem.
In particular, they  prove that the system is exactly controllable in the unit sphere of the Sobolev space $H^7$ (implying in particular  approximate controllability  in $L^2$). One of the most interesting corollaries of this result  
is  exact controllability between eigenstates.

A different result is given in \cite{Boscain_Adami}, where  
adiabatic methods are used to 
prove approximate controllability for systems having conical eigenvalue crossings in the space of controls.

Another controllability result has been proved by  Mirrahimi in \cite{mirra-solo} using Strichartz estimates and concerns 
approximate controllability for a certain class of systems such that $\Omega=\R^d$ and
whose drift Hamiltonian has mixed spectrum (discrete and continuous).

The aim of 
the present paper is to prove a general approximate controllability 
result  for a large class of systems for which the drift Hamiltonian  
 $H_0$ has discrete spectrum.
Our main assumptions  are
 that
the spectrum of  $H_0$  
 satisfies a non-resonance condition and that 
 $H_1$ couples each pair of distinct eigenstates of $H_0$.
Such assumptions happen to be generic in a suitable sense, as it will be discussed in a forthcoming paper.

We then apply the approximate controllability  result to two classical examples, namely the harmonic oscillator
 and the 3D potential well,  for suitable controlled potentials.

Our method is new in the framework of quantum control and relies on finite-dimensional techniques applied to the Galerkin 
approximations. 
A difficult point is to deduce properties of the original infinite-dimensional system from its finite-dimensional approximations. For 
the Navier--Stokes  equations this program 
 was successfully conducted by Agrachev and Sarychev in the seminal paper \cite{Navier-Stokes} (see also \cite{shirikyan,rodrigues}). 

A key ingredient of the proof is a  time reparametrization that inverts the roles of  $H_0$ and $H_1$ as drift and control operator.  This operation is crucial since it permits to exploit for the Galerkin approximation the techniques developed  in  
\cite{agrachev_chambrion} for finite-dimensional systems on compact semisimple Lie groups. The passage from the controllability properties of the Galerkin approximations to those of the infinite-dimensional system heavily relies on the fact that the dynamics preserve the Hilbert sphere.

A feature of our method  is that the infinite-dimensional system inherits, in a suitable sense,  
controllability results for  the group of unitary 
transformations from those of the Galerkin approximations.
This permits to extract controllability properties for the {\it density matrix}.  
Let us stress that, as it happens in finite dimension, controllability properties for the 
density matrix cannot in general be deduced from those of the wave function (see for instance \cite{albertini}).

The paper is organized as follows. In Section \ref{s-math-f} we present the general functional analysis setting and we state our main result (Theorem~\ref{main}) for the control system  \r{eqeq}. In Section  \ref{bible} we show how this result applies to the Schr\"odinger equation \r{eqeq2} when  $\Omega$ is both bounded or unbounded. Section \ref{scheme} contains the proof of Theorem~\ref{main} and an estimate of the minimum time for approximately steering the system between two given states, that holds even if the system itself is not approximately controllable. In Section \ref{s-density} we extend  Theorem~\ref{main} to the controlled evolution of the density matrix (Theorem~\ref{densities}). Finally in Section \ref{examples} we show how Theorem~\ref{main} and Theorem~\ref{densities}
can be applied to specific cases.  
In particular, we show how to get controllability results  even in cases in which  $V$ 
does not satisfy the required
non-resonance hypothesis, using perturbation arguments.

\section{Mathematical framework and statement of the main result}
\label{s-math-f}
Hereafter $\N$ denotes the set of strictly positive integers. 
Definition~\ref{CDS} below provides the abstract mathematical framework 
that will be used to formulate and prove 
the controllability results  later applied to 
the \Sch\ equation \r{eqeq2}.
The hypotheses under which  \r{eqeq2} fits the abstract framework are  
discussed in Section~\ref{bible}. 

\begin{defn}\label{CDS}
Let $\H$ be a complex Hilbert space and $U$ be a subset of $\R$.
Let $A,B$ be two, possibly unbounded, operators on $\H$ with values in $\H$ and denote by $D(A)$ and $D(B)$ their domains. 
The control system $(A,B,U)$
is the formal controlled equation 
\begin{equation} \label{main_EQ}
\frac{d \psi}{dt}(t)=A \psi(t) +u(t) B\psi(t),\ \ \ \  \ \ u(t)\in U.
\end{equation}
We say that $(A,B,U)$ is a skew-adjoint
\diag\ 
 control system 
if the following  conditions are satisfied: 
(H1) $A$ and $B$ are 
skew-adjoint, 
(H2) there exists an orthonormal basis  $(\phi_n)_{n\in\N}$ of $\H$ made of
eigenvectors of $A$, (H3) $\phi_n\in D(B)$ for every $n\in\N$.
\end{defn}

In order to give a meaning to the evolution equation \r{main_EQ}, at least when $u$ is constant,
we should ensure that the sum $A+u B$ is well defined.
The standard notion of sum of operators seen as quadratic forms (see \cite{davies}) is not 
always applicable under the sole hypotheses (H1), (H2), (H3).
An adapted definition of $A+u B$ 
can nevertheless be given as follows: 
hypothesis (H3) guarantees that the sum  $A+u B$ is well defined on $V=\spann\{\phi_n\mid n\in\N\}$.  
Any skew-Hermitian operator 
$C:V\to \H$ 
admits a unique skew-adjoint extension ${\mathcal E}(C)$. 
We identify $A+ u B$ with ${\mathcal E}(A|_{V}+u B|_V)$. 

Let us notice that when $A+ u B$ is well defined as 
sum of quadratic forms and is skew-adjoint 
then the two definitions of sum coincide. This happens in particular for the \Sch\ equation \r{eqeq2}
in most physically significant situations (see Section~\ref{bible}).

 A crucial consequence of what precedes 
is that
 for every $u\in U$
 the 
skew-adjoint operator $A+uB$ generates 
a group of unitary transformations
$e^{t(A+u B)}:\H\to \H$.
In particular, the unit sphere $\s$  of $\H$ satisfies $e^{t(A+u B)}(\s)=\s$ for every $u\in U$ and every $t\geq 0$.

Due to the dependence of the domain $D(A+uB)$ on $u$, 
the solutions of \r{main_EQ}  cannot in general be defined in classical (strong, mild or weak) sense.
Let us mention that, in some relevant cases in which the spectrum of $A$ has a nontrivial continuous component the solution can be defined as in \cite{pierfi,rodni} by means of Strichartz estimates.

We will say that the {\it solution of \r{main_EQ}} with initial condition $\psi_0\in\H$
 and corresponding to the 
  piecewise constant control $u:[0,T]\to U$ is the 
  curve $t\mapsto \psi(t)$
 defined by  
\be\label{solu}
\psi(t)=e^{(t-\sum_{l=1}^{j-1} t_l)(A +u_j B)}\circ e^{t_{j-1}(A +u_{j-1} B)}\circ \cdots \circ e^{t_1(A+u_1  B)}(\psi_0),
\ee
where $\sum_{l=1}^{j-1} t_l\leq t<\sum_{l=1}^{j} t_l$ and 
$u(\tau)=u_j$ if $\sum_{l=1}^{j-1} t_l\leq \tau<\sum_{l=1}^{j} t_l$.
Notice that such a $\psi(\cdot)$  satisfies, for every $n\in \N$ and almost every $t\in [0,T]$,
 the differential equation
\be\label{very_weak}
\frac d{d t}\la \psi(t),\phi_n\ra=-\la \psi(t),(A+u(t)B)\phi_n\ra.
\ee

\begin{rem}\label{rem}
The notion of solution introduced above 
makes sense in very degenerate situations 
and can be enhanced  when $B$ is bounded. Indeed, 
well-known results assert that in this case
if $u\in L^1([0,T],U)$ then  
there exists a unique weak (and mild) solution $\psi\in \mathcal{C}([0,T],\H)$ 
which coincides with the curve \r{solu} when $u$ is piecewise constant. 
Moreover, if $\psi_0\in D(A)$ and $u\in \mathcal{C}^1([0,T],U)$ then 
$\psi$ is differentiable and it is a strong solution of \r{main_EQ}.
(See \cite{bms} and references therein.)
\end{rem}

\begin{defn}\label{controllability}
Let $(A,B,U)$ be a skew-adjoint  \diag\  control system. We say that $(A,B,U)$ is approximately controllable
if for every $\psi_0,\psi_1\in \s$ and every $\eps>0$ there exist $k\in \N$, $t_1,\dots,t_k>0$ and $u_1,\dots,u_k\in U$ such that 
$$\|\psi_1-e^{t_k(A+u_k B)}\circ \cdots \circ e^{t_1(A+u_1 B)}(\psi_0)\|<\eps.$$
\end{defn}

Let, for every $n\in\N$, $i \lb_n$ denote the eigenvalue of $A$ corresponding to $\phi_n$ ($\lb_n\in\R$).
The main result of the paper is 
the following. 

\begin{theorem}\label{main}
Let $\delta>0$ and 
$(A,B,(0,\delta))$ be a skew-adjoint  \diag\  control system. If the elements of the sequence
$(\lb_{n+1}-\lb_{n})_{n\in\N}$ are $\Q$-linearly independent
and if 
$\la B \phi_n, \phi_{n+1}\ra\ne 0$ for every $n\in \N$,
then $(A,B,(0,\delta))$ is approximately controllable. 
\end{theorem}

Recall that the elements of the sequence
$(\lb_{n+1}-\lb_{n})_{n\in\N}$ are said to be $\Q$-linearly independent 
if for every $N\in \N$ and $(q_1,\dots,q_N)\in\Q^N\smallsetminus\{0\}$ 
one has $\sum_{n=1}^N q_n (\lb_{n+1}-\lb_{n})\ne 0$.

The condition $\la B \phi_n, \phi_{n+1}\ra\ne 0$, preferred 
here for the easiness of its expression, can be replaced by a weaker one (namely, \r{connect}), as detailed in Remark~\ref{connectedness}.
\section{Discrete-spectrum \Sch\ operators}\label{bible}

The aim of this section is to recall some classical results on \Sch\ operators. 
 In particular we list here, among the numerous situations 
studied in the literature,  
some well-known sufficient conditions 
guaranteeing that 
the controlled \Sch\ equation \r{eqeq2}
satisfies the assumptions of 
Definition~\ref{CDS}. 

\begin{theorem}
[{\cite[Theorem 1.2.2]{henrot}}]
\label{the-bounded}
Let $\Omega$ be an open and bounded subset of $\R^d$ and  $V\in L^\infty(\Omega,\R)$. 
Then $-\Delta+V$, with Dirichlet boundary conditions, 
is a self-adjoint  operator with compact resolvent.
In particular  $-\Delta+V$ 
has discrete  spectrum and 
admits a family of eigenfunctions 
in  $H^2(\Omega,\R)\cap H^1_0(\Omega,\R)$ 
which forms an orthonormal basis of 
$L^2(\Omega,\C)$. 
\end{theorem}

{
\begin{theorem}[{\cite[Theorems XIII.69 and XIII.70]{reed_simon}}]\label{the-unbounded}
Let $\Omega=\R^d$  and  
$V\in L^1_{\mathrm{loc}}(\R^d,\R)$ be bounded from below and such that
$$\lim_{|x|\to \infty}V(x)=+\infty.$$
Then $-\Delta+V$, defined as a sum of quadratic forms, 
is a self-adjoint operator with compact resolvent.
In particular  $-\Delta+V$ 
has discrete  spectrum and 
admits a family of eigenfunctions 
in $H^2(\R^d,\R)$ which forms an orthonormal basis 
of $L^2(\R^d,\C)$.
Moreover, for
every  eigenfunction $\phi$ of  $-\Delta+V$ 
and for every $a>0$, $x\mapsto e^{a\|x\|}\phi(x)$ belongs to $L^2(\R^d,\C)$.
\end{theorem}
}

In the following, we call controlled \Sch\ equation the partial differential equation
$$i\frac{\partial \psi}{\partial t}(t,x) = (-\Delta +V+uW)\psi(t,x)$$
where $\psi:I\times \Omega\to\C$, $\Omega$ is an open subset of $\R^d$, $I$ is a subinterval of $\R$ and, in the case in which $\Omega$ is bounded, $\psi|_{I\times\partial\Omega}=0$. 
The correct functional analysis framework for this equation is specified below.  

The following corollary, which is a straightforward consequence of the results recalled above, states that the assumptions of Definition~\ref{CDS} are fulfilled by the 
operators appearing in the controlled \Sch\ equation under natural hypotheses.

\begin{corol}\label{booo}
Let  $\Omega$ be an open subset of $\R^d$, $V,W$ be two real-valued functions defined on $\Omega$, and $U$ be a subset of $\R$. 
Assume either that (i) $\Omega$ is bounded, $V,W$ belong to $L^\infty(\Omega,\R)$
or that (ii) $\Omega=\R^d$, $V,W$ belong to $L^1_{\mathrm{loc}}(\R^d,\R)$, 
the growth of $W$ at infinity is at most exponential and, for every $u\in U$, $\lim_{\|x\|\to+\infty}(V(x)+uW(x))=+\infty$ and $\inf_{x\in\R^d}(V(x)+uW(x))>-\infty$.
Let $\H$ be equal to $L^2(\Omega,\C)$ and $D(A)$ be equal to
$H^2(\Omega,\C)\cap H^1_0(\Omega,\C)$ in case (i) and to $H^2(\Omega,\C)$ in case (ii). Let, moreover, $A$ be the differential operator $-i(-\Delta+V)$ and $B$ be the multiplication operator $-iW$.  
Then $(A,B,U)$ is a skew-adjoint  discrete-spectrum  control system, called the \emph{controlled \Sch\ equation} associated with $\Omega,\,V,\,W$ and $U$.  
\end{corol}

Since the controlled \Sch\ equation is a skew-adjoint discrete-spectrum  control system, it makes sense to apply Theorem~\ref{main} to it. The result is the following theorem.

\begin{theorem}\label{main-Sch}
Let $\Omega,\,V,\,W$ and $U$ satisfy  
one of the hypotheses (i) or (ii) of Corollary~\ref{booo}. 
Denote by $(\lb_k)_{k\in\N}$ the sequence of eigenvalues of $-\Delta+V$ and by $(\phi_k)_{k\in\N}$ an orthonormal basis of $L^2(\Omega,\C)$ of corresponding real-valued eigenfunctions. 
Assume, in addition to (i) or (ii), that $U$ contains the interval $(0,\delta)$ for some $\delta>0$, that 
the elements of $(\lambda_{k+1}-\lb_k)_{k\in\N}$ are $\Q$-linearly independent, and that $\int_\Omega W(x)\phi_k\phi_{k+1}\,dx\ne0$ for every $k\in\N$. Then 
the {controlled \Sch\ equation} associated with $\Omega,\,V,\,W$ and $U$
is approximately controllable.
\end{theorem}
As stressed just after the statement of Theorem~\ref{main}, 
the condition $\int_\Omega W(x)\phi_k\phi_{k+1}\,dx\ne0$
could be replaced by a weaker one (see  Remark~\ref{connectedness}).

\section{Proof of Theorem~\ref{main}}\label{scheme}\label{s-proof}

The proof of Theorem~\ref{main} is split in 
several steps.
First, in Section~\ref{TR} the controllability problem is transformed, thanks to a time-reparameterization,
 into an equivalent one where $A$ and $B$ play the role of controlled dynamics and drift, respectively. 
 Then, in Section~\ref{GalAp}, 
 we prove a controllability result for the 
 Galerkin approximations  of this equivalent system. In Section~\ref{lifting} we show how to 
 lift the  controllability properties from a Galerkin approximation to an higher-dimensional one. 
 Section~\ref{china} makes the link between finite-dimensional and infinite-dimensional controllability properties
 and completes the proof. 

Finally, in Section~\ref{lowb}, as a byproduct of the arguments of the proof, we get  a lower bound on the minimum  steering time.

\subsection{Time-reparameterization}\label{TR}
First
remark that, if $u\ne 0$, 
$e^{t(A+u B)}=e^{tu((1/u)A+ B)}$.
Theorem~\ref{main} is therefore equivalent to the following property: 
if the elements of the sequence
$(\lb_{n+1}-\lb_{n})_{n\in\N}$ are $\Q$-linearly independent
and if 
$\la B \phi_n, \phi_{n+1}\ra\ne 0$ for every $n\in \N$, then 
for every  $\delta,\eps>0$ and 
every  $\psi_0,\psi_1\in \s$ 
there exist $k\in \N$, $t_1,\dots,t_k>0$ and $u_1,\dots,u_k>\delta$ such that 
\be
\|\psi_1-e^{t_k(u_k A +B)}\circ \cdots \circ e^{t_1(u_1 A+ B)}(\psi_0)\|<\eps.
\ee
In other words, the system for which the roles of $A$ and $B$ as drift and controlled field are inverted, namely, 
\begin{equation}\label{timed}
\frac{d \psi}{dt}(t)=u(t) A\psi(t) + B\psi(t),\ \ \ \  \ \ u(t)\in U,
\end{equation}
is approximately controllable provided that the control set $U$ contains a half-line.
The notion of solution of \r{timed} corresponding to a piecewise constant control function is 
defined 
as in \r{solu}.

\subsection{Controllability of the Galerkin approximations}\label{GalAp}
Let, for every $j,k\in\N$,  $b_{jk}=\la B \phi_j, \phi_k\ra$ and 
$a_{jk}=\la A \phi_j, \phi_k\ra=i \lb_j \delta_{jk}$. 
(Recall that $\lb_j\in\R$ and $\{i\lb_j\mid j\in\N\}$ is the spectrum of $A$.)
Define, for every $n\in\N$,
the two complex-valued $n\times n$ matrices  $A^{(n)}=(a_{jk})_{1\leq j,k\leq n}$ and $B^{(n)}=(b_{jk})_{1\leq j,k\leq n}$. 
The Galerkyn approximation of \r{timed} at order $n$ 
(with respect to the basis $(\phi_k)_{k\in\N}$)
is the finite-dimensional control system
$$
\frac{d x}{dt}=u A^{(n)}x +B^{(n)}x,\ \ \ \ \ \ x\in\s_n,\ \ \ \ \ \ u>\delta,\eqno{(\Sigma_n)}
$$
where $\s_n$ denotes the unit sphere of $\C^n$.
Notice that the system is well defined since, by construction, 
$A^{(n)}$ and $B^{(n)}$ are skew-Hermitian matrices.

We say the ($\Sigma_n$) is {\it controllable} if for every $x_0,x_1\in \s_n$ there exist $k\in \N$, $t_1,\dots,t_k>0$ and $u_1,\dots,u_k>\delta$ such that 
$$
x_1=e^{t_k(u_k A^{(n)} +B^{(n)})}\circ \cdots \circ e^{t_1(u_1 A^{(n)}+ B^{(n)})}x_0.
$$

We recall that a $n\times n$ matrix $C=(c_{jk})_{1\leq j,k\leq n}$ is said to be connected  
 if for every pair of indices $j,k\in\{1,\dots,n\}$ there exists a finite sequence $r_1,\dots,r_l\in\{1,\dots,n\}$ such that $c_{j r_1}c_{r_1 r_2}\cdots c_{r_{l-1}r_l}c_{r_l k}\ne 0$. (In the literature connected matrices are 
 sometimes called {\it einfach}, or irreducible, or inseparable.)
The following proposition is
in the spirit of the controllability results obtained in \cite{agrachev_chambrion} and \cite{turinici}.

\newcommand{\An}{\mathfrak{A}}
\newcommand{\Bn}{\mathfrak{B}}
\newcommand{\an}{\alpha}
\newcommand{\bn}{\beta}
\begin{prop} \label{k-lab}
Let $\An=(\an_{jk})_{j,k=1}^n$,  $\Bn=(\bn_{jk})_{j,k=1}^n$ be two skew-symmetric $n\times n$ matrices and assume that $\An$ is diagonal and
$\Bn$ is connected. 
Assume moreover that $|\an_{jj}-\an_{kk}|\ne|\an_{ll}-\an_{mm}|$ if $\{j,k\}\ne\{l,m\}$. 
Then  the control system
$(\overline{\Sigma}): \dot x=u\An x+\Bn x$  
is controllable in $\s_n$ with piecewise constant controls $u:\R\to \overline{U}$, provided that $\overline{U}$ contains at least two points. 
\end{prop} 

\begin{proof}
For every $1\leq j,k\leq n$ let $e_{jk}$ be the $n\times n$ 
matrix whose entries are all equal to zero except the one at line $j$ and column $k$ which is equal to $1$. 

Define for every $p\in\N$ the iterated matrices commutator
 $M_p=\mathrm{ad}^p_{\An}(\Bn)$. (Recall the usual notation $\ad_X(Y)=[X,Y]=XY-YX$ for the adjoint operator associated with $X$.)
A simple induction on $p$ 
shows that 
the matrix $M_p$
has the expression
$$M_p=\sum_{l,m=1}^n  (\an_{ll}-\an_{mm})^p  \bn_{lm}e_{lm}.$$
 
 Fix two indices $j\neq k$ such that $1\leq j,k \leq n$ and $\bn_{jk}\ne 0$. 
 Since, by hypothesis,  $(\an_{jj} - \an_{kk})^2 \neq (\an_{ll} - \an_{mm})^2$
 as soon as $\{j,k\} \neq \{l,m\}$,
 there exists some polynomial $P_{jk}$ with real coefficients such that
 $P_{jk}\big((\an_{jj} - \an_{kk})^2\big)=1$ and $P_{jk}\big((\an_{ll} - \an_{mm})^2\big)=0$ for
 all $\{l,m\} \neq \{j,k\}$, $1\leq l,m \leq n$. 
 
 Let us define $(c_h)_h$ as the coefficients of $P_{jk}$, i.e.,  $P_{jk}(X)=\sum_{h=0}^{d} c_{h} X^{h}$.
 Define moreover the
 matrix $N_{jk}=\sum_{h=0}^{d} c_{h} M_{2h}$. 
 By
 construction 
 $N_{jk}=\sum_{l,m=1}^n \bn_{lm} e_{lm} P_{jk}\big((\an_{ll} - \an_{mm})^2\big)=\bn_{jk}e_{jk}-\overline{\bn_{jk}}e_{kj}$.
 Therefore, 
the commutator $[\An,N_{jk}]$ 
is equal to $(\an_{jj}-\an_{kk})(\bn_{jk}e_{jk}+\overline{\bn_{jk}}e_{kj})$
and so the Lie algebra generated by $\An$ and $\Bn$
contains the two elementary anti-Hermitian matrices 
$E_{jk}=e_{jk}-e_{kj}$ and $F_{jk}=i (e_{jk}+e_{kj})$.

 Notice now that, for every $1\leq j,k,h,m\leq n$, $e_{jk}e_{hm}=\delta_{kh}e_{jm}$ and therefore 
\brs
 {[E_{jk},E_{km}]}&=&E_{jm}+\delta_{km}E_{kj}+\delta_{kj}E_{mk},\\
 {[E_{jk},F_{jk}]}&=&2i(e_{jj}-e_{kk}).
\ers 
 It follows from the definition of connected matrix and the relation 
  $[\An,E_{jk}]=i(\an_{kk}-\an_{jj})F_{jk}$ that  the Lie algebra generated by $\An$ and $\Bn$ contains the matrices $E_{jk}$, $F_{jk}$ and $i(e_{jj}-e_{kk})$ for every 
  $j\ne k$. Therefore
 \begin{equation}
 \mathfrak{su}(n)\subseteq \Lie(\An,\Bn).
 \label{sun} 
 \end{equation}
 \newcommand{\pr}{{\mathcal P}}
  Fix $\bar x\in\s_n$ and consider the submersion $\pr:SU(n)\to \s_n$ defined by $\pr(g)=g \bar x$. Since
  $$T_{\pr(g)} (\s_n)=\pr_*(T_g SU(n))=\pr_*(\mathfrak{su}(n) g)=\mathfrak{su}(n)g \bar x=\mathfrak{su}(n)\pr(g),$$
   then the evaluation at $x=\pr(g)$ of the Lie    algebra generated by $\An$ and $\Bn$ contains the whole space $T_x\s_n$. Since for any $u\in\overline{U}$ and any $t\in\R$ the flow $e^{t(u\An+ \Bn)}:\s_n\to\s_n$ is volume-preserving then $(\overline{\Sigma})$ is controllable (see \cite[Cor. 8.6, Prop. 8.14, Th. 8.15]{book2}).
    \end{proof}
 
 The condition $b_{j,j+1}\ne 0$ for every $j\in \N$ appearing in Theorem~\ref{main} clearly ensures that every matrix $B^{(n)}$ is connected.
  Proposition~\ref{k-lab}, applied to $A^{(n)}=\An$ and 
 $B^{(n)}=\Bn$, implies therefore that $(\Sigma_n)$ is controllable.
 
\begin{rem}\label{connectedness}
 In the following we can replace the  assumption 
that $b_{j,j+1}\ne 0$ for every $j\in \N$
 with the weaker one that $B^{(n)}$ is frequently connected, that is,
 \be\label{connect}
\forall j\in\N,\exists k\geq j\mid B^{(k)}\mbox{ is connected.}
 \ee
 Notice, as a partial counterpart, that if there exists a nonempty and proper subset $\Xi$ of $\N$ such that for every $j\in \Xi$ and $k\in \N\smallsetminus \Xi$ the coefficient $b_{jk}$ is equal to zero (i.e., the infinite-dimensional matrix $(b_{lm})_{l,m\in\N}$ is non-connected) 
 then the control system \r{main_EQ} is not approximately controllable. Indeed, the subspace
 $\spann\{\phi_k\mid k\in \Xi\}$ is invariant for the dynamics of $A+u B$ for every $u\in U$ and
 has nontrivial (invariant) orthogonal. 
  \end{rem}

 \subsection{Approximate controllability in higher-dimensional projections}\label{lifting}
Fix $\delta,\eps>0$ and $\psi_0,\psi_1\in \s$. 
For every $n\in \N$, let $\Pi_n:\H\to \H$ be the orthogonal projection 
on the 
space $
\spann(\phi_1,\dots,\phi_n)$ and $\PPi_n:\H\to\C^n$ be the map that associates to an element of $\H$ the vector of its first $n$ coordinates with respect to the basis $(\phi_m)_{m\in \N}$.
Choose $n$ such that 
\begin{equation}\label{newequation}
\|\psi_j-\Pi_n(\psi_j)\|<\eeta\ \ \ \mbox{ for $j=0,1$.}
\end{equation} 
 
Thanks to \r{connect} we can assume, without loss of generality, that ($\Sigma_n$) is controllable. Let $u:[0,T]\to (\delta,\infty)$ 
be the piecewise constant control driving $\xi_0/\|\xi_0\|$ to $\xi_1/\|\xi_1\|$ where $\xi_j=\PPi_n(\psi_j)$ for $j=0,1$.

 
Let  $\mu>0$ be a 
constant which will be chosen later small enough, 
depending on $T$, $n$, and $\eps$.
Notice that for every $j\in\N$ the hypothesis that $\phi_j$ belongs to $D(B)$ implies that 
the sequence $(b_{jk})_{k\in \N}$ is in $l^2$. 
It is therefore possible to choose 
$N\geq n$ such that 
\be
\sum_{k> N}|b_{jk}|^2<\mu,\quad \mbox{ for every }\ j=1,\dots,n.
\label{mubastardo}
\ee

If $t\mapsto X(t)$ is a solution of ($\Sigma_N$) corresponding to a control function $U(\cdot)$, then 
$t\mapsto e^{-V(t)A^{(N)}}X(t)=Y(t)$, where $V(t)=\int_0^t U(\tau)d\tau$, 
is a solution of
$$\dot Y(t)=e^{-V(t)A^{(N)}} B^{(N)} e^{V(t)A^{(N)}}Y(t). \eqno{(\Theta_N)}$$

Let us represent the matrix  $e^{-v(t)A^{(N)}} B^{(N)} e^{v(t)A^{(N)}}$, where $v(t)=\int_0^t u(\tau)d\tau$, in block form as follows
\beq
e^{-v(t)A^{(N)}} B^{(N)} e^{v(t)A^{(N)}}=\lp\ba{cc}M^{(n,n)}(t) & M^{(n,N-n)}(t) \\  M^{(N-n,n)}(t) & M^{(N-n,N-n)}(t) \ea\rp\,,
\label{blocks}
\eeq
where the superscripts indicate the dimensions of each block. 
\begin{claim}\label{roba}
There exists a sequence of piecewise constant 
control functions $u_k:[0,T]\to (\delta,\infty)$ such that
the sequence of matrix-valued curves
$$t\mapsto M_k(t)=e^{-v_k(t)A^{(N)}} B^{(N)} e^{v_k(t)A^{(N)}},$$
where $v_k(t)=\int_0^t u_k(\tau)d\tau$, 
converges to
\[t\mapsto M(t)=\lp\ba{cc} M^{(n,n)}(t) &0_{n\times (N-n)}\\0_{(N-n)\times n}& M^{(N-n,N-n)}(t)\ea\rp\]
in the following integral sense
\be\label{integral}
\int_0^t M_k(\tau)d\tau \to \int_0^t M(\tau)d\tau
\ee
as $k\to \infty$ uniformly with respect to $t\in [0,T]$.
\end{claim}
\proof
We will prove the claim taking $v(\cdot)$ piecewise constant, since every piecewise affine function can be approximated arbitrarily well 
in the $L^\infty$ 
topology 
by piecewise constant functions  and 
because the map  
associating to $v(\cdot)$  the curve  
$t\mt \int_0^t M(\tau)d\tau$
is continuous with respect to the $L^\infty$ topology (taken both in 
its domain and its codomain).

Assume that $v(\cdot)$ is constantly equal to $w\in\R$ on $[t_1,t_2]$.
Since 
$\lb_2-\lb_1,\dots,\lb_N-\lb_{N-1}$  are $\Q$-linearly independent,
then for every $s_0\in\R$ the 
curve 
$$(s_0,\infty)\ni s\mt ((\lb_1-\lb_2)s,\dots,(\lb_1-\lb_{N})s)$$
projects onto a dense subset of the torus $\R^{N-1}/2\pi\Z^{N-1}$. 
Thus, there exist two sequences $w^{(m)}\nearrow+\infty$ and $z^{(m)}\nearrow+\infty$
such that
\newcommand{\mmod}{\mathrm{mod}}
\[
\ba{rclll}
 (\lambda_1-\lambda_j)w^{(m)}\ (\mmod~2 \pi) &\longrightarrow&
(\lambda_1-\lambda_j) w\ (\mmod~2 \pi)&&\mbox{for }2\leq j \leq N,\\
(\lambda_1-\lambda_j)z^{(m)}\ (\mmod~2 \pi)&\longrightarrow& (\lambda_1-\lambda_j) w\ (\mmod~2 \pi)&& \mbox{for }2\leq j \leq n,\\
(\lambda_1-\lambda_j)z^{(m)}\ (\mmod~2 \pi)&\longrightarrow&(\lambda_1-\lambda_j) w + \pi\ (\mmod~2 \pi)&&\mbox{for }n+1 \leq j \leq N,
\ea
\]
 as $m$ tends to infinity.
In particular the sequence of matrices $e^{-w^{(m)} A^{(N)}} B^{(N)} e^{w^{(m)}A^{(N)}}$ converges to
$e^{-wA^{(N)}} B^{(N)} e^{wA^{(N)}}$ as $m$ goes to infinity, while the sequence $e^{-z^{(m)} A^{(N)}} B^{(N)} e^{z^{(m)}A^{(N)}}$ converges, following the notations 
introduced in
\eqref{blocks},  to 
$$\lp\ba{cc} M^{(n,n)} & -M^{(n,N-n)}\\ -M^{(N-n,n)} & M^{(N-n,N-n)}\ea\rp,$$
where we dropped the dependence on $t$ of the different sub-matrices since  $v$ is constant on $[t_1,t_2]$. 

Fix $\bar\delta>\delta$. Consider a sequence $(\bar v_k)_{k\in\N}$ in $\R_+$ (whose role will be clarified later)
and define, for every $k\in\N$, a finite increasing  sequence $(\theta^{k}_l)_{l=0,\ldots,k}$ with $\theta^{k}_0=\bar v_k\,$, $\theta^{k}_{l+1}\geq\theta^{k}_l+\bar\delta (t_2-t_1)/k^2$ for $0\leq l<k$, and such that, for $l>0$, 
$\theta^{k}_l$ belongs to $(w^{(m)})_{m\in\N}$ 
if $l$ is odd and to $(z^{(m)})_{m\in\N}$ if $l$ is even.
Define, moreover, 
$$\tau_j=t_1+\frac{j-1}k(t_2-t_1),$$
for $j=1,\dots,k+1$.

Consider the 
continuous
function $v_{k}$ uniquely defined on $[t_1,t_2]$ by the conditions
$$
\left\{\ba{rcll}
v_{k}(t_1)&=&\bar v_k,&\\
v_{k}(\tau_i+\frac{t_2-t_1}{k^2})&=&\theta^k_i&\mbox{for }i=1,\dots,k,\\
\dot{v}_k(t)&=&\bar\delta&\mbox{if }t\in\cup_{i=1}^k (\tau_i+\frac{t_2-t_1}{k^2},\tau_{i+1}),
\\
\ddot v_k(t)&=&0&\mbox{if }t\in\cup_{i=1}^k (\tau_i,\tau_i+\frac{t_2-t_1}{k^2}).\ea\right.
$$

\begin{figure}
\begin{center}
\input{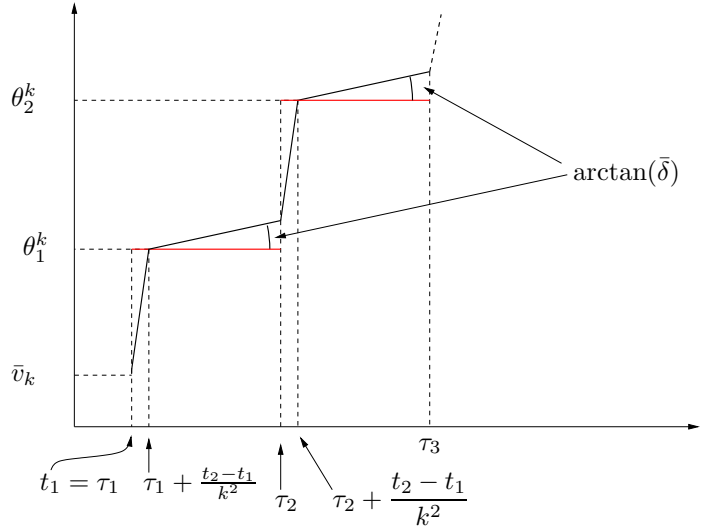}
\caption{The functions $v_k$ and $\tilde v_k$}
\end{center}
\end{figure}
(See  Fig.~1.)
Define $\tilde v_k$ as the piecewise constant function that coincides with $\theta^{k}_i$ on $[\tau_i,\tau_{i+1}]$. On each interval $[\tau_i+ (t_2-t_1)/{k^2},\tau_{i+1}]$ the difference between $v_{k}$ and $\tilde v_{k}$ is bounded in absolute value by $\bar\delta (t_2-t_1)/k $.
Therefore, 
$$\sup\left\{\left\|e^{-v_{k}(t)A^{(N)}} B^{(N)} e^{v_{k}(t)A^{(N)}}-e^{-\tilde v_{k}(t)A^{(N)}} B^{(N)} e^{\tilde v_{k}(t)A^{(N)}}\right\|\mid t\in\cup_{i=1}^k \left[\tau_i+\frac{t_2-t_1}{k^2},\tau_{i+1}\right]\right\}
$$
goes to zero 
as $k$ goes to infinity.

Since $\|e^{-\nu A^{(N)}} B^{(N)} e^{\nu A^{(N)}}\|$ is uniformly bounded with respect to $\nu\in \R$ and the measure of  $\ \cup_i [\tau_i,\tau_i+\frac{T}{k^2}]\ $ goes to $0$ as $k$ goes to infinity, we have
$$\int_{t_1}^t \big( e^{-v_{k}(\tau)A^{(N)}} B^{(N)} e^{v_{k}(\tau)A^{(N)}} - e^{-\tilde v_{k}(\tau)A^{(N)}} B^{(N)} e^{\tilde v_{k}(\tau)A^{(N)}}\big) d\tau \stackrel{k\to\infty}{\longrightarrow}0\quad\mbox{ uniformly on }[t_1,t_2]\,.$$
Moreover, by definition of 
the sequences  
$\theta^{k}_i$, $w^{(m)}$, and $z^{(m)}$, one has  
$$\int_{t_1}^t e^{-\tilde v_{k}(\tau)A^{(N)}} B^{(N)} e^{\tilde v_{k}(\tau)A^{(N)}} d\tau \stackrel{k\to\infty}{\longrightarrow}\int_{t_1}^t M(\tau)d\tau\quad\mbox{ uniformly on }[t_1,t_2]\,,$$
and therefore $e^{-v_{k}(t)A^{(N)}} B^{(N)} e^{v_{k}(t)A^{(N)}}$ converges in integral sense to $M(t)$ on $[t_1,t_2]$.


Finally, 
construct $u_k$ as follows: 
for $t_1=0$ define 
$u_k$ on $(t_1,t_2)$
as the derivative of $v_k$ (defined almost everywhere), 
where the 
$v_k$'s correspond to the sequence of initial conditions  
$\bar v_k=0$ for every $k\in\N$. 
Then, on the second interval on which $v(\cdot)$ is constant,  use as a new set of initial conditions 
for the approximation procedure 
the values
$\bar v_k=v_k(t_2)$ and define again  $u_k$ 
as the derivative of $v_k$. 
Iterating the procedure on the finite set of intervals covering $[0,T]$ 
on which $v(\cdot)$ is constant we obtain the required approximating sequence of piecewise constant control functions. 
\EOP

\subsection{Approximate controllability for the infinite-dimensional system}\label{china}
Let $u_k$ and $M_k$ be defined as in Claim~\ref{roba}. 
The resolvent $R_k(t,s):\C^N\to \C^N$ of the linear time-varying equation
$$\dot Y=M_k(t) Y,$$
converges, uniformly with respect to $(t,s)$, to the resolvent $R(t,s):\C^N\to \C^N$ of
$$\dot Y=M(t) Y.$$
(See, for instance, \cite[Lemma 8.10]{book2}.) Notice that $R(t,s)$ preserves the norms of both the vector formed by the first $n$ coordinates and the one formed by the last $N-n$.

Let, for every $k\in\N$,  $\psi^k$ be the solution of \r{timed} corresponding to $u_k$. We have the following approximation property. 

\definecolor{rosso}{rgb}{1,0,0}

\newcommand{\rosso}[1]{{\bf {\color{rosso} #1}}}

\begin{claim}\label{mm}
For $k$ large enough,
 \be
 \|\Pi_n(e^{- v_k(T)A}\psi^k(T))-\Pi_n(e^{- v(T)A}\psi_1)\|<2\eps,\label{2eps}
 \ee
 where $\eps$ is the positive constant which has been fixed at the beginning of Section~\ref{lifting}.
 \end{claim}
 \begin{proof} 
Define $ q^k(t)=e^{-i v_k(t) A}\, \psi^k(t) $.
According to \r{very_weak} (more precisely, its counterpart for equation \r{timed}), the components 
$q^k_j(t)=e^{-i \lb_j v_k(t)}\la \psi^k(t),\phi_j\ra$ 
of $q^k(t)$ with respect to the basis of eigenvectors of $A$ 
satisfy for almost every $t\in[0,T]$
\be\label{wind}
\dot q_j^k(t)=\sum_{l=1}^\infty b_{jl}e^{i(\lb_l-\lb_j)v_k(t)}q_l^k(t).
\ee
Therefore, the curves $P^k(t)=(q^k_1(t),\dots,q^k_n(t))^T$ and
$Q^k(t)=(q^k_{n+1}(t),\dots,q^k_N(t))^T$ satisfy
\[
\lp\ba{c} {\dot P}^k(t)\\{\dot Q}^k(t)\ea\rp=M_k(t)\lp\ba{c} P^k(t)\\Q^k(t)\ea\rp+\lp\ba{c} H^k(t)\\I^k(t)\ea\rp 
\]
with $\|H^k\|_\infty <\sqrt{n\mu}$ (see \r{mubastardo}) and $\|I^k\|_\infty < C$ for $C=C(N)$ large enough.

Hence
\[
\lp\ba{c} {P}^k(t)\\{Q}^k(t)\ea\rp=R_k(t,0)\PPi_N(\psi_0)+
\int_0^t R_k(s,t) \lp\ba{c} H^k(s)\\I^k(s)\ea\rp ds. 
\]
Denote by $\PPi_n^N$ the projection of $\C^N$ on its first $n$ coordinates and let
$$L^k(t)=\PPi_n^N\lp\int_0^t R_k(s,t) \lp\ba{c} H^k(s)\\I^k(s)\ea\rp ds\rp.$$
Since $R_k$ converges uniformly to $R$ and 
the latter
preserves the norm of the first $n$ components,
we know that, for $k$ large, $\|L^k\|_\infty< 2 T \sqrt{n\mu}$.
Moreover,
$R_k(t,0)\PPi_N(\psi_0)$ converges uniformly to 
$R(t,0)\PPi_N(\psi_0)$.
In particular, according to the definition of $R$, 
$\PPi_n^N(R_k(t,0)\PPi_N(\psi_0))$ converges uniformly
to the solution of ($\Theta_n$) corresponding to the control $u$ and starting from $\PPi_n(\psi_0)=\xi_0$. 

Since $u$ drives system $(\Sigma_n)$ from 
$\xi_0/\|\xi_0\|$ to $\xi_1/\|\xi_1\|$, 
then it steers system ($\Theta_n$) from 
$\xi_0$ to $e^{-v(T)A^{(n)}}\xi_1(\|\xi_0\|/\|\xi_1\|)$.
 Therefore, 
 $$\left\|P^k(T)-e^{-v(T)A^{(n)}}\xi_1\frac{\|\xi_0\|}{\|\xi_1\|}\right\|
 <3 T \sqrt{n\mu},$$ 
 if $k$ is large enough. 
 Let us fix $\mu$ small enough in order to have 
 $$3 T \sqrt{n\mu}<\eps.$$
Then, 
\brs
\nonumber\|\Pi_n(e^{- v_k(T)A}\psi^k(T))-\Pi_n(e^{- v(T)A}\psi_1)\|
&\leq&
\left\|\Pi_n(e^{- v_k(T)A}\psi^k(T))-\Pi_n(e^{- v(T)A}\psi_1)\frac{\|\xi_0\|}{\|\xi_1\|}\right\|\\
\nonumber&&\;+\left \|\Pi_n(e^{- v(T)A}\psi_1)\frac{\|\xi_0\|}{\|\xi_1\|}-\Pi_n(e^{- v(T)A}\psi_1)\right\|\\
\nonumber&=&\left\|P^k(T)-e^{-v(T)A^{(n)}}\xi_1\frac{\|\xi_0\|}{\|\xi_1\|}\right\|\\
\nonumber&&\;+\|\Pi_n(e^{- v(T)A}\psi_1)\|\frac{|\, \|\xi_1\|-\|\xi_0\|\,|}{\|\xi_1\|}\\
\nonumber&<&3 T \sqrt{n\mu}+|\, \|\xi_1\|-\|\xi_0\|\,|\\
&<&
2\eeta,
\ers
provided that $k$ is large enough. 
\end{proof}
As a consequence of Proposition~\ref{mm}, 
for $k$ large enough the moduli of the first $n$ components of 
$\psi^k(T)$ are close to those of the first $n$ components of $\psi_1$. 
The proposition  below 
will be used to show
that their phases 
can also 
be made as close as required by applying a
suitable control on an arbitrarily small time interval.

\newcommand{\tlb}{\tilde \lb}
\newcommand{\tA}{\tilde A}
\newcommand{\tB}{\tilde B}
\newcommand{\tde}{\tilde \delta}
\newcommand{\teps}{\tilde \eps}

\begin{prop}\label{phase}
For every $\teps>0$, every $\tlb_1,\dots,\tlb_n\in \R$, and every 
$s_1\in\R$ 
there exist $s_2>s_1$ and $w\in \R^n$ with $\|w\|\leq\teps$ such that $\tlb_i s_2\equiv w_i\mod{2\pi}$ for every $i=1,\ldots,n$. 
As a consequence, 
given a 
 skew-adjoint  \diag\  control system  $(\tA,\tB,(0,\tde))$,  
for every $v_1\in\R$ and every 
 $\tau>0$ small enough
there exists  a 
constant
control function $\tilde u:[0,\tau]\to (\tde,+\infty)$ 
such that every  trajectory $\tilde\psi(\cdot)$ 
of $\dot \psi=u \tA \psi+\tB\psi$ corresponding to $\tilde u(\cdot)$ satisfies $\|\Pi_n(\tilde\psi(\tau))-\Pi_n(e^{v_1 \tA}\tilde\psi(0))\|\leq \teps$ (where $\Pi_n$ denotes the orthogonal projection on the space spanned by the first $n$ eigenvectors of $\tA$).
\end{prop}
\proof\ 
The first part of the 
statement is a simple application of the Poincar\'e recurrence theorem.
Indeed, since the dynamics $s\mapsto x_0+s(\tlb_1,\dots,\tlb_n)$ on the $n$-dimensional torus preserve volumes and distances,
then the constant vector field $(\tlb_1,\dots,\tlb_n)$ is recurrent
at
every point of the torus, and in particular at the origin $x_0=0$.   
Therefore any neighborhood ${\cal N}$ of the origin is sent, after a suitably long time (which can be assumed to be larger than $s_1$), to another neighborhood of the origin 
isometric to ${\cal N}$.
Taking ${\cal N}$ equal to the ball of radius $\teps$ centered at the origin,  
the first part of the claim is proven.

In order to conclude the proof, 
fix a piecewise constant
control function $\tilde u:[0,\tau]\to (\tde,+\infty)$
and a solution $\tilde\psi$ of $\dot \psi=u \tA \psi+\tB\psi$
corresponding to $\tilde u$. Set  
$$\tilde q(t)=e^{-\int_0^t \tilde u(s)ds \tA}\tilde\psi(t)$$
and notice that, according to \r{very_weak},
$|\dot {\tilde q}_j(t)|\leq \|\tB\tilde\phi_j\|$
 for every $j\in\N$, where $\tilde\phi_j$ denotes the $j$-th eigenvector of $\tA$ and ${\tilde q}_j=\la {\tilde q},\tilde\phi_j\ra$.
 Therefore, $\|\Pi_n(\tilde q(\tau))-\Pi_n(\tilde q(0))\|\leq C\tau$ for some positive constant $C$ independent of $\tilde u$ and of $\tilde \psi$. 
Then
\brs
\|\Pi_n(\tilde\psi(\tau))-\Pi_n(e^{v_1 \tA}\tilde\psi(0))\|&=&\|\Pi_n(e^{-v_1 \tA}\tilde\psi(\tau))-\Pi_n(\tilde\psi(0))\|\\
&=&\|\Pi_n(e^{\lp \int_0^\tau\tilde u(t)dt-v_1\rp \tA}\tilde q(\tau))-\Pi_n(\tilde q(0))\|\\
&\leq&\|\Pi_n(e^{\lp \int_0^\tau\tilde u(t)dt-v_1\rp \tA}\tilde q(\tau))-\Pi_n(\tilde q(\tau))\|+C\tau.
\ers
Fix $\tau<\teps/(2C)$ so that
\be\label{bracca}
\|\Pi_n(\tilde\psi(\tau))-\Pi_n(e^{v_1 \tA}\tilde\psi(0))\|\leq \|\Pi_n(e^{\lp \int_0^\tau\tilde u(t)dt-v_1\rp \tA}\tilde q(\tau))-\Pi_n(\tilde q(\tau))\|+\frac\teps2.
\ee 
Notice that 
$$\Pi_n(e^{\lp \int_0^\tau\tilde u(t)dt-v_1\rp \tA}\tilde q(\tau))
=
\mathrm{diag}\lp e^{i\tlb_1\lp \int_0^\tau\tilde u(t)dt-v_1\rp},\dots,e^{i\tlb_n\lp \int_0^\tau\tilde u(t)dt-v_1\rp}\rp
\Pi_n(\tilde q(\tau)).$$
The first part of the claim ensures the existence of $v_2$ arbitrarily 
large 
such that if  
\be\label{stracca}
\int_0^\tau\tilde u(t)dt-v_1=v_2
\ee 
then the norm of the matrix 
$$\mathrm{diag}\lp e^{i\tlb_1\lp \int_0^\tau\tilde u(t)dt-v_1\rp},\dots,e^{i\tlb_n\lp \int_0^\tau\tilde u(t)dt-v_1\rp}\rp-\mathrm{Id}_{n}$$
is smaller than $\teps/2$. 
Take $v_2$ 
large enough to satisfy
$(v_1+v_2)/\tau>\tde$. Then $\tilde u\equiv (v_1+v_2)/\tau$
satisfies \r{stracca} and, because of \r{bracca}, 
$$
\|\Pi_n(\tilde\psi(\tau))-\Pi_n(e^{v_1 \tA}\tilde\psi(0))\|\leq \teps,
$$ 
 independently of $\tilde\psi$. \hfill$\Box$ 

To conclude the proof of Theorem~\ref{main} 
we extend the interval of definition of 
the control function 
$u_k$ introduced above
by taking $u_k(T+t)=\tilde u(t)$ for $t\in[0,\tau]$ where 
$\tilde u$ is the control 
obtained by applying Proposition~\ref{phase} with
$v_1=v(T)-v_k(T)$, $\tA=A$, $\tB=B$, $\teps=\eps$.
Then, for $k$ large enough,  the corresponding trajectory
$\psi^k:[0,T+\tau]\to {\s}$ satisfies
$\|\Pi_n(\psi^k(T+\tau))-\Pi_n (\psi_1)\|<3\eps$
and therefore, due to (\ref{newequation}), 
$$\|\Pi_n(\psi^k(T+\tau))-\psi_1\|<4\eps.$$
It is now enough to show that
$\|\psi^k(T+\tau)-\Pi_n(\psi^k(T+\tau))\|$ can be made arbitrarily  small  by choosing a suitably small $\eps$. To this aim we 
notice that the inequality 
$\|\Pi_n(\psi^k(T+\tau))\|>1-4\eps$ implies, for $\eps<1/4$, that
 $\|\psi^k(T+\tau)-\Pi_n(\psi^k(T+\tau))\|^2<1-(1-4\eps)^2=8\eps-16\eps^2$
and this concludes the proof of Theorem~\ref{main}.
 \EOOP


\subsection{Lower bound on the steering time}
\label{lowb}

In  this section we prove a lower bound on the steering time for a 
skew-adjoint discrete-spectrum control system
without assuming that it satisfies the hypotheses of Theorem~\ref{main} nor
any other controllability assumption.

\begin{prop}\label{bornetime}
Let $(A,B,(0,\delta))$ be a skew-adjoint discrete-spectrum control 
system.  
Fix $\psi_{0},\psi_{1}$ in $\s$ and $\eps>0$. 
Then if a piecewise constant control $u:[0,T_u] \rightarrow (0,\delta)$
steers system (\ref{main_EQ})  from $\psi_0$ to an $\eps$-neighborhood of $\psi_1$, then
\begin{equation} \label{eq-minor-time}
T_u \geq \frac{1}{\delta} \mbox{sup}_{k \in \N} \frac{\big||\langle \phi_k,\psi_{0} \rangle| - |\langle \phi_k,\psi_{1} \rangle|\big| -\epsilon}{\| B \phi_k \|},
\end{equation}
where $(\phi_k)_{k\in \N}$ denotes the orthonormal basis of eigenvectors of $A$. 
\end{prop}
\begin{proof}
Fix an initial condition $\psi_{0}$ in $\s$, a piecewise constant control $u:[0,T_u] \rightarrow (0,\delta)$, and denote by $\psi^{u}:[0,T_u] \rightarrow \H$ the corresponding solution of the system (\ref{main_EQ}) satisfying $\psi^u(0)=\psi_{0}$.
 
Write $u$ as 
$u(t)=\sum_{j=0}^n  u_j \chi_{[t_j,t_{j+1})}(t)$ where $0=t_1<t_2<\cdots<t_{n+1}=T_u$ 
and $u_1,\ldots,u_n$ 
belong  
 to $(0,\delta)$. 
 In the spirit of Section~\ref{TR}, 
associate to $u$ the piecewise constant control $\nu:[0, T_\nu]\to \R$  given by $\nu(t)=\sum_{j=0}^n  \nu_j \chi_{[\tau_j,\tau_{j+1})}(t)$ with 
$\nu_j=1/{u_j}$ for all $j=1,\dots,n$ and 
$\tau_j$ defined by induction as
$\tau_1=0$,  
 $\tau_{j+1}=\tau_j + (t_{j+1}-t_j) u_j$ for $j\geq 1$.

Define $\psi^\nu:[0,T_\nu] \rightarrow \H$ as the solution of system (\ref{timed}) corresponding to $\nu$ and satisfying $\psi^\nu(0)=\psi_{0}$.    Define by $m_k=|\langle \psi^\nu , \phi_k \rangle |$ the modulus of the $k^{\mathrm{th}}$ coordinate of $\psi^\nu$.

By definition $m_k$ is absolutely continuous and 
equation (\ref{wind}) implies that 
$$\dot{m}_k \leq \sum_{j=1}^{\infty} |b_{jk}| m_j \leq \left(\sum_{j=1}^{\infty} |\langle B \phi_j, \phi_k \rangle|^2\right)^{1/2}= \| B \phi_k \|.$$ 
Applying the mean value theorem, one gets  
\begin{equation} \label{eq-major-composante} \big||\langle \psi^\nu(0),\phi_k \rangle| -|\langle \psi^\nu(T_\nu),\phi_k \rangle|\big|\leq T_\nu \| B \phi_k \|.
\end{equation} 
 
Notice that $T_\nu=\sum_{j=1}^{n} (t_{j+1}-t_j) u_j \leq (t_n-t_1) \delta = T_u \delta$, that is, \begin{equation} \label{eq-major-Tv} 
 T_u \geq \frac{1}{\delta} T_\nu. 
 \end{equation}
Since, by assumption,
$|\psi^\nu(T_\nu)-\psi_{1}|<\epsilon$, 
then  (\ref{eq-major-composante}) implies 
$$T_\nu \geq \mbox{sup}_{k \in \N} \frac{\big||\langle \phi_k,\psi_{0} \rangle| - |\langle \phi_k,\psi_{1} \rangle|\big| -\epsilon}{\| B \phi_k \|}.$$ 
Plugging this last inequality into (\ref{eq-major-Tv}), we 
obtain \r{eq-minor-time}.
\end{proof}

We insist on the fact that this result is valid
 whenever system (\ref{main_EQ}) is or is not approximately controllable.

\begin{rem} When $B$ is bounded,  the same estimate as above is valid for other classes of controls 
(not only 
piecewise constant functions but also 
measurable bounded or locally integrable)
 as soon as we can define a unique solution of  system \r{main_EQ} that satisfies \r{very_weak}. See Remark~\ref{rem}.
\end{rem}

\begin{rem}
It follows from \r{eq-minor-time} that, in general,
approximate  controllability does not imply finite-time approximate controllability. Indeed, if $B \phi_k$ tends to $0$ as $k$ goes to infinity, then for every $T>0$ the attainable set at time $T$ from a given point $\psi_0$ is not dense in $\s$ since for every $\eps\in (0,1)$, for $k$ large enough, $\phi_k$ is not $\eps$-approximately attainable from    $\psi_0$ in time $T$.

\end{rem}

\section{Controllability for density matrices}
\label{s-density}

\newcommand{\UU}{\mathbf{U}}
\newcommand{\VV}{\mathbf{V}}

\subsection{Physical motivations}
 A {\it density matrix} (sometimes called density operator) is a non-negative, self-adjoint operator of trace class \cite[Vol.~I]{reed_simon} on a Hilbert space. The trace of a density matrix is normalized to one. As a consequence of the definition  a density matrix  is a compact operator (hence with discrete spectrum) and can always be written as a weighted sum of projectors, 
\begin{eqnarray}
\rho=\sum_{j=1}^\infty P_j \varphi_j \varphi_j^\ast,
\label{density-m}
\end{eqnarray}
where  $P_j\in [0,1]$, $\sum_j P_j=1$, and 
$\varphi_j \varphi_j^\ast$ is the {\it orthogonal projector} on the space spanned by  $\varphi_j$ with  $\varphi_j^\ast(\cdot)=\la \varphi_j,\cdot\ra$. Here $\{ \varphi_j\}_{j\in \N}$ is a set of normalized vectors not necessarily orthogonal.

The density matrix is used to describe the evolution of systems whose initial wave function is not known precisely, but only with a certain probability, or when one is dealing with an ensemble of identical systems that cannot be prepared precisely in the same state. More precisely (\ref{density-m}) describes a system whose state is known to be $\varphi_j$ with probability $P_j$, $j\in\N$. Given an observable $A$ (i.e. a self-adjoint operator, for instance the drift Hamiltonian) the mean value   of $A$ is  $\mbox{Tr}(\rho A)=\sum_{j=1}^\infty P_j \la \varphi_j,A\varphi_j\ra$, where $\la\varphi_j,A\varphi_j\ra$ represents the mean value of the observable $A$ in the  state $\varphi_j$. 
When for some $k\in{\N}$ we have $P_k=1$ and $P_j=0$  for every  $j\neq k$, one says that $\rho$ describes a {\it pure state}, otherwise one says that $\rho$ describes a {\it mixed state}. In the case of pure states, the physical description via the density matrix is equivalent to the one via the wave function. Notice that for a pure state Tr$(\rho^2)=1$ while for a  mixed state one has Tr$(\rho^2)<1$.

Without loss of generality it is possible to  require that  $\{\varphi_j\}_{j\in\N}$  is  an orthonormal basis  (i.e. a basis of normalized eigenvectors of $\rho$). In this case $\{P_j\}_{j\in \N}$ is the spectrum of $\rho$. 

The time evolution of the density matrix 
is determined by the evolutions of the states $\varphi_j$, namely
\begin{eqnarray}
\rho(t)=\UU(t) \rho(0) \UU^{\ast}(t)
\end{eqnarray}
where $\UU(t)$ is the operator of temporal evolution (the resolvent) and $\UU^{\ast}(t)$ its adjoint.
Notice that the spectrum of $\rho(t)$ is constant along the motion.

\subsection{Statement of the result}

Fix $\delta>0$ and let $(A,B,(0,\delta))$ be  a skew-adjoint  \diag\  control system on a Hilbert space $\H$,  $(\varphi_j)_{j\in \N}$ an orthonormal basis of $\H$ (not necessarily of eigenvectors of $A$), $\{ P_j \}_{j\in N}$ a sequence of non-negative numbers such that $\sum_{j=1}^\infty P_j=1$, and denote by $\rho$ the density matrix
$$
\rho=\sum_{j=1}^{\infty}P_j \varphi_j {\varphi_j}^\ast.
$$
\begin{defn}
Two density matrices  $\rho_0$ and $\rho_1$ are said to be {\it unitarily equivalent} if there exists a unitary transformation $\UU$ of $\H$ such that  $\rho_1=\UU \rho_0  \UU^\ast$.
\end{defn}
Obviously the controllability question for the evolution of the density matrix makes sense only for pairs  $(\rho_0,\rho_1)$ of  initial and final density matrices that are  unitarily equivalent.
Notice that this  is a quite strong assumption, since it
 implies that the eigenvalues of  $\rho_0$ and $\rho_1$ are the same. Controllability results in the case of density matrices that are not unitarily  equivalent   have been obtained in the case of open systems (i.e. systems evolving under a suitable nonunitary evolution) in the finite-dimensional case. See for instance \cite{altafini}.

Next section is devoted to the proof of the following theorem.

\begin{theorem}\label{densities}
Let $\rho_0$ and $\rho_1$ be two  unitarily equivalent  density matrices. Then, under the hypotheses of Theorem \ref{main},  for every $\eps>0$ there exists a piecewise constant control steering the density matrix 
from $\rho_0$ 
$\eps$-approximately to $\rho_1$ i.e. 
there exist $k\in \N$, $t_1,\dots,t_k>0$ and $u_1,\dots,u_k\in (0,\delta)$ such that setting
$\VV=e^{t_k(A+u_k B)}\circ \cdots \circ e^{t_1(A+u_1 B)}$, one has 
$\|\rho_1-    \VV\rho_0\VV^\ast \|<\eps$, 
where $\|  \cdot\|$ denotes the operator norm on $\H$.
\end{theorem}

\begin{rem}
As Theorem~\ref{main-Sch} is a particularization of Theorem~\ref{main} to the controlled \Sch\ equation, the hypotheses of 
Theorem~\ref{main-Sch} imply $\eps$-approximate controllability 
of the corresponding  density matrix. 
\end{rem}

\subsection{Proof of Theorem~\ref{densities}}
The proof uses the notations of Section~\ref{s-proof}. 
As noticed in Section \ref{TR}, 
the theorem can be restated in terms of 
the evolution of the density matrix corresponding to the control system $\dot \psi=(u A+B)\psi$, $u\in (\delta,+\infty)$. 

Fix $\rho_0$ and $\rho_1$ unitarily equivalent and let $\UU$ be such that 
$\rho_1=\UU\rho_0\UU^\ast$. Write 
$$\rho_0=\sum_{j=1}^\infty P_j \varphi_{0,j} {\varphi_{0,j}}^\ast,$$
 with $(P_j)_{j\in N}$ a sequence of non-negative numbers whose sum is one, and $(\fhi_{0,j})_{j\in N}$ an orthonormal basis of $\H$. Then 
 $$\rho_1=\sum_{j=1}^\infty P_j \varphi_{1,j} {\varphi_{1,j}}^\ast,$$
 with  $\fhi_{1,j}=\UU\fhi_{0,j}$ 
 for every $j\in \N$.

Choose $\eps>0$. Let $m\in\N$ be such that
$$\sum_{j>m}P_j<\eps.$$
The idea is to follow the strategy applied in the proof of Theorem~\ref{main} in order to simultaneously approximately steer $m$ copies of system $(A,B,(0,\delta))$ from $\fhi_{0,j}$ to $\fhi_{1,j}$, $j=1,\dots,m$.

Let $\eta>0$ be a small constant depending on $m$ and $\eps$, to be fixed later. 
There exists $n=n(\eta)>m$ such that, for every $j=1,\dots,m$ and for $k=0,1$,
$$\|\fhi_{k,j}-\Pi_{n}\fhi_{k,j}\|<\eta.$$
By construction, when $\eta$ gets small, the two families
$(\Pi_{n}\fhi_{k,j})_{j=1}^{m}$, $k=0,1$,  tend to two  orthonormal families. 
Hence, there exists a matrix $M$ in $SU(n)$ such that 
\be\label{e-e-e}
\| M(\PPi_{n}\fhi_{0,j})-\PPi_{n}\fhi_{1,j}\|<\eps
\ee
for $j=1,\dots,m$ provided that $\eta$ is small enough (and, consequently, $n$ is large enough).

Without loss of generality we may assume that $B^{(n)}$ is connected. 
Claim~\ref{k-lab} can be extended to the following result.

\newcommand{\Aa}{A^{(n)}}
\newcommand{\Bb}{B^{(n)}}
\newcommand{\tAa}{\tilde A^{(n)}}
\newcommand{\tBb}{\tilde B^{(n)}}

 \begin{claim}
The control system 
\be\label{ode-group}
\dot{g}= (u \Aa+\Bb)g,\ \ \ \ \ \ g\in U(n),
\ee
is controllable 
in the following sense: for any $g_0$, $g_1$ in $U(n)$, there exists a unitary complex number $e^{i \theta}$ with $0\leq \theta\leq 2\pi/n$, a time $T>0$ and a piecewise constant 
function $u:[0,T] \rightarrow (\delta,+\infty)$ such that the solution 
$g^u: [0,T] \rightarrow U(n)$ of \r{ode-group} 
with initial condition  $g^u(0)= g_0$
satisfies $e^{i \theta} g^u(T) =g_1$. 
\end{claim}

\begin{proof} 
Let us first assume that at least one among $\Aa$ and $\Bb$ has nonzero trace and hence does not belong to $\mathfrak{su}(n)$. In this case the inclusion \r{sun}, with $\Aa=\An$ and $\Bb=\Bn$, implies that
$\Lie(\Aa,\Bb)=\mathfrak{u}(n).$
Classical controllability results for right invariant systems on compact Lie groups  (see \cite{jur,such}) 
ensure that the attainable set from $g_0$ of \r{ode-group} coincides with $U(n)$ so that the claim holds with $\theta=0$.

It remains to consider the case in which the traces of $\Aa$ and $\Bb$ are zero, i.e. $\Aa$ and $\Bb$ belong to $\mathfrak{su}(n)$. In this case \r{sun} implies that
$\Lie(\Aa,\Bb)=\mathfrak{su}(n),$ 
and therefore the attainable set from $g_0$ of \r{ode-group} coincides with $g_0 SU(n)$, the set of matrices of $U(n)$ having the same determinant as $g_0$.
Given a target $g_1$  there exists $\vartheta\in[0,2\pi]$ such that $\det(g_0)=e^{-i\vartheta}\det(g_1)=\det(e^{-i\frac{\vartheta}{n}} g_1)$. Hence the claim holds true with $\theta=\vartheta/n$.
\end{proof}

 Let $T>0$, $u:[0,T]\to (\delta,+\infty)$ and $0\le \theta\le 2\pi/n$ be such that the control $u$ steers system \r{ode-group} from $I_n$ 
 to $e^{i\theta} M$. 
 Notice that, without loss of generality, $2\pi/n<\eps$.
 
 Let $\mu>0$ be a small constant to be fixed later.
 Fix $N\in\N$ such that 
 $$ \| (b_{jl})_{l>N}\|_{l^2}<\mu$$
 for every $j=1,\dots,n$.
Let us apply Claim~\ref{roba} to the control function $u$ and denote
by
 $(u_k)_{k\in \N}$ the sequence of piecewise constant control functions obtained in this way.
 Write, moreover, $v(t)=\int_0^t u(\tau)d\tau$ and $v_k(t)=\int_0^t u_k(\tau)d\tau$.
For every $k\in \N$ write $u_k$ as
$$u_k(t)=\sum_{j=1}^{p_k} w^k_j \chi_{[t_j^k,t_{j+1}^k)}(t),\ \ \ \ t\in[0,T]$$ 
with $0= t_1^k\leq\cdots\le t_{p_k}^k= T$
and
denote by $\VV_k$ the unitary transformation
$$\VV_k=e^{\lp t_{p_k}^k-t_{p_k-1}^k\rp\lp w_{p_k-1}^k A+B\rp}\circ \cdots\circ 
e^{(t_2^k-t_{1}^k)(w_{1}^k A+B)}.$$
For every $j=1,\dots,m$,
\brs
 \| \Pi_n\lp e^{(v(T)-v^k(T))A}\VV_k \fhi_{0,j}\rp-\Pi_n\lp \fhi_{1,j}\rp\|
 &\le&\| \PPi_n\lp e^{(v(T)-v^k(T))A}\VV_k \fhi_{0,j}\rp-e^{i\theta}M\lp \PPi_n\fhi_{0,j}\rp \|\\
 &&+\| 
e^{i\theta} M\lp \PPi_n\fhi_{0,j}\rp-M\lp \PPi_n\fhi_{0,j}\rp\|\\
 &&+\| 
 M\lp \PPi_n\fhi_{0,j}\rp-
 \PPi_n\lp 
 \fhi_{1,j}\rp\|.
 \ers
The same computations as in Section~\ref{china} 
(cf. \r{2eps}) show that, for every $j=1,\dots,m$,
$$\| \PPi_n\lp e^{(v(T)-v^k(T))A}\VV_k \fhi_{0,j}\rp-e^{i\theta} M\lp \PPi_n\fhi_{0,j}\rp \|\le 2\eps$$
for $\mu$ small and $k$ large enough. 
Since $0\leq \theta\leq 2\pi/n<\eps$, then, for every $j=1,\dots,m$,
$$\| 
e^{i\theta} M\lp \PPi_n\fhi_{0,j}\rp-M\lp \PPi_n\fhi_{0,j}\rp\|\leq |e^{i\theta}-1|<\eps.$$
Hence, because of \r{e-e-e}, 
for $k$ large enough, for every $j=1,\dots,m$,
$$ \| \Pi_n\lp e^{(v(T)-v^k(T))A}\VV_k \fhi_{0,j}\rp-\Pi_n\lp \fhi_{1,j}\rp\|<4\eps.$$
Applying Proposition~\ref{phase} we can, up to the  extension of  $u^k$ 
to a piecewise constant control defined on  a larger interval, 
assume  that 
$$ \| \Pi_n\lp\VV_k \fhi_{0,j}\rp-\Pi_n\lp \fhi_{1,j}\rp\|<5\eps,$$
for every $j=1,\dots,m$. Therefore,
\newcommand{\VVV}{\VV_k}
\brs
\| \VVV\rho_0\VVV^\ast-\rho_1\|&=&\|\sum_{j=1}^\infty P_j\lp (\VVV \fhi_{0,j})(\VVV \fhi_{0,j})^\ast-\fhi_{1,j} \fhi_{1,j}^\ast\rp\|\\
&\leq&\|\sum_{j=1}^m P_j\lp (\VVV \fhi_{0,j})(\VVV \fhi_{0,j})^\ast-\fhi_{1,j} \fhi_{1,j}^\ast\rp\|+2\eps\\
&\leq&\sum_{j=1}^m P_j\lp \| \VVV \fhi_{0,j}\|+\| \fhi_{1,j}\|\rp\| \VVV \fhi_{0,j}-\fhi_{1,j}\| +2\eps\\
&\leq&2(5\eps)+2\eps\ =\ 12\eps,
\ers
provided that $k$ is large enough.

\section{Examples}\label{examples}

\subsection{Perturbation of the spectrum} \label{sec-perturbation}

The scope of Section \ref{examples} is to show how 
the general controllability results obtained 
in the previous sections
can be applied in specific cases.
In particular, we want to show  how the
conditions on the spectrum of the \Sch\ operator 
appearing in the hypotheses of Theorem~\ref{main-Sch}
can be checked in practice. 

Let us adopt the notations of Section~\ref{bible} for the
domain $\Omega$, the wave function $\psi$, and the uncontrolled and  controlled potentials $V$ and $W$.
Throughout this section we assume that
one of the hypotheses (i) or (ii) of Corollary~\ref{booo} holds true. Thus, $(A,B,U)$ is a well-defined  controlled \Sch\ equation, where $A=-i(-\Delta+V)$ and $B=-i W$.

The study of the examples 
below is based on the simple idea
that, even if the 
hypotheses of Theorem~\ref{main-Sch} 
are not satisfied 
by the operators $A$ and $B$,
one can anyway ensure that they hold true 
for $A_\mu=-i(-\Delta+V+\mu W)$ and $B_\mu=-i
W$ for some $\mu$ in the interior of $U$.
This is enough to conclude 
that the system $\dot \psi=A\psi+uB\psi$, $u\in U$, is approximately controllable,
since the replacement of $(A,B)$ by $(A_\mu,B_\mu)$ 
corresponds to a reparameterization of $U$ that sends $u$ into a new control $u-\mu\in U-\mu$ and $V$ into $V+ \mu W$. 
Although the spectrum 
of $A_\mu$ is not in general 
explicitly  computable, we can nevertheless deduce 
 some crucial properties about it by applying standard perturbation arguments.
 Theorem~\ref{the-perturbation} 
  recalls, in a simplified version suitable for our purposes, some classical perturbation results  
 describing the dependence on $\mu$ of the spectrum of 
 $-\Delta +V+\mu W$. (See \cite[Chapter VII, Remark 4.22]{katino}, \cite[\S II.10, Theorem~1]{Rellich} and also \cite{Albert}.)
\begin{theorem}  \label{the-perturbation}
Let $U$ be an open interval containing zero.
Assume either that (i) $\Omega$ is bounded, $V,W$ belong to $L^\infty(\Omega)$
or that (ii) $\Omega=\R^d$, $V$ belongs to $L^1_{\mathrm{loc}}(\R^d)$, $W$ belongs to $L^\infty(\R^d)$, 
$\lim_{\|x\|\to+\infty}V(x)=+\infty$ and $\inf_{x\in\R^d}V(x)>-\infty$.
 In both cases (i) and (ii) assume that each eigenvalue of the \Sch\ operator $-\Delta+V$ is 
 simple. Denote by $(\lb_k)_{k\in\N}$ the 
 sequence of eigenvalues of $-\Delta+V$ and by $(\phi_k)_{k\in\N}$ the 
 corresponding eigenfunctions. 
Then, for any $k$ in $\mathbf{N}$, there exist 
two analytic curves $\Lambda_k:U \to \mathbf{C}$
and $\Phi_k:U \to L^2(\Omega)$    such that:
\begin{itemize}
\item $\Lambda_k(0)=\lambda_k$ and $\Phi_k(0)=\phi_k$;
\item for any $\mu$ in $U$,  $(\Lambda_k(\mu))_{k\in\N}$ is the 
family of eigenvalues of $\Delta -V + \mu W$ counted according to their multiplicities
and $(\Phi_k(\mu))_{k\in\N}$ is an orthonormal basis of
corresponding eigenfunctions;
\item ${\Lambda'_k}(0)=\int_{\Omega} W(x) |\phi_k(x)|^2 dx$.
\end{itemize}
\end{theorem}

We check below that if the derivatives ${\Lambda'_k}(0)$ are 
$\mathbf{Q}$-linearly independent then for almost every $\mu\in U$ 
the eigenvalues of $-\Delta+V+\mu W$ are $\mathbf{Q}$-linearly independent. 
This fact is used in the following to apply Theorem~\ref{main-Sch} to situations 
in which the uncontrolled \Sch\ operator has a resonant spectrum.  

Recall that, in the notations of Section~\ref{scheme}, for any pair of integers $j,k\in \N$, 
\be\label{def-bjk}
b_{jk}=\int_{\Omega} W(x)  \phi_j(x) \phi_k(x) dx.
\ee
In particular, ${\Lambda'_k}(0)=\int_{\Omega} W(x) |\phi_k(x)|^2 dx$ is equal to $b_{kk}$.

\begin{prop}\label{pro-pro}
Let $U$ be an open interval containing zero and assume that
$\Omega$, $V$ and $W$ satisfy one of the hypotheses (i) or (ii) of Theorem~\ref{the-perturbation} and that the eigenvalues of $-\Delta+V$ are simple. If the elements of the sequence 
$(b_{kk})_{k\in\N}$
are $\mathbf{Q}$-linearly independent, then for almost every $\mu$ in $U$ the elements of  $({\Lambda_k}(\mu))_{k \in \N}$ are  $\mathbf{Q}$-linearly independent.  
\end{prop}

\begin{proof}
Let $l\in\N$ and $z=(z_1,\dots,z_l)\in\Q^l$. 
 Denote by $\Upsilon_z$ the subset of elements $\mu$ in $U$ such that $\sum_{j=1}^l z_j \Lambda_j(\mu)=0$. Since each $\mu\mt {\Lambda_k}(\mu)$ is an analytic function, then 
 $\Upsilon_z$ is either equal to $U$ or to a countable subset of $U$. 
 Since $b_{11}=\Lambda'_1(0),\dots,b_{ll}=\Lambda'_l(0)$ are $\mathbf{Q}$-linearly independent, then $\Upsilon_z=U$ if and only if $z=0$. 
Hence, the union ${\bf \Upsilon}=\cup_{l \in \mathbf{N}} \cup_{z\in\Q^l,\;z\ne0}\Upsilon_z$ has Lebesgue measure zero, since it is countable. 
By construction, if $\mu$ does not belong to ${\bf \Upsilon}$, the elements of  $({\Lambda_k}(\mu))_{k \in \N}$ are $\mathbf{Q}$-linearly independent. 
\end{proof}

The other crucial hypothesis of 
Theorem~\ref{main-Sch} is 
that $b_{j,j+1}\ne0$ for every $j\in\N$ (or, more generally, that
$B^{(n)}=(b_{jk})_{j,k=1}^n$ is frequently connected, see Remark~\ref{connectedness}). By the same analyticity argument as above one checks that either 
such hypothesis is always false or it is true for almost every $\mu\in U$.

\begin{corol}\label{complete}
Let $U$ be an open interval containing zero and assume that
$\Omega$, $V$ and $W$ satisfy one of the hypotheses (i) or (ii) of Theorem~\ref{the-perturbation} and that the eigenvalues of $-\Delta+V$ are simple. 
Assume moreover that the elements of the sequence 
$(b_{kk})_{k\in\N}$
are $\mathbf{Q}$-linearly independent 
and that $B^{(n)}$ is frequently connected. 
Then 
the controlled 
\Sch\ equation associated with $\Omega$, $V$, $W$ and $\tilde U$ is approximately controllable for every $\tilde U\subset U$ with nonempty interior.
\end{corol}

\subsection{1D harmonic oscillator}

In this section we 
study the Schr\"{o}dinger equation describing the evolution of the controlled one-dimensional  harmonic oscillator,
\be\label{1Dosc}
i \frac{\partial \psi}{\partial t}(t,x) = - \frac{\partial^{2} \psi}{\partial x^2}(t,x) + \left ( x^2 - u(t) W(x) \right ) \psi(t,x),
\ee
where $\psi$ is the wave function depending on the time  $t $ and on 
a space variable $x \in \mathbf{R}=\Omega$. 
Recall that $u(\cdot)$ is a piecewise-continuous function with values in a subset $U$ of $\R$.
Notice that the potential corresponding to the uncontrolled 
\Sch\ operator is $V(x)=x^2$.
The control system \r{1Dosc} has been studied, among others,
by Mirrahimi and Rouchon who proved its non-controllability 
in the case where $W$ is the identity function (see \cite{mira_rouch}).

As a consequence of Theorem~\ref{the-unbounded}, the spectrum of $-\Delta +V$ is discrete.
Its explicit expression  is 
$$\left \{\lambda_k={2k+1}\mid k\geq 0 \right \},$$
and therefore $\lambda_{k+1}-\lambda_k$ are  $\Q$-linearly dependent.
  Each $\lambda_k$ is a simple eigenvalue 
  whose corresponding eigenfunction is 
\be\label{e-funx}
\phi_k(x)= \frac{1}{\sqrt{k! 2^k \sqrt{\pi}}} e^{-\frac{x^2}{2}} H_k(x)
\ee
where $H_k(x)=(-1)^k e^{x^2}\frac{d^k}{dx^k} e^{-x^2}$ is the $k^{\mathrm{th}}$ Hermite polynomial.

In order to apply Corollary~\ref{complete} we would like first of all to ensure that the elements
\be\label{deri}
b_{kk}=\frac{(-1)^k}{{k! 2^k \sqrt{\pi}}}\int_\R W(x)H_k(x)\frac{d^k}{dx^k} e^{-x^2}d x,\ \ \ \ \ \ k\geq 0,
\ee
are $\Q$-linearly independent. Notice that for $W(x)=x$ (i.e., the non-controllable case pointed out by Mirrahimi and Rouchon), since each function $\phi_k^2$ is even, $b_{kk}=\int W\phi_k^2=0$. 

The existence of
controlled potentials $W$ for which  the elements of $(b_{kk})_{k\in\N}$
are  $\Q$-linearly independent can be easily inferred from the linear independence of the functions $\phi_k^2$.
The proposition below provides some explicit $W$
with such a property (and such that the corresponding \Sch\ equation is controllable).  
The potentials $W$ will be chosen in $L^\infty(\R)$ and therefore, as already remarked in Section~\ref{bible}, the corresponding solutions in the sense \r{solu} coincide with mild or strong solutions, depending on the regularity of the initial condition.  

\begin{prop}
(1) If $W$ is even, then system  \r{1Dosc} is not approximately controllable. 
(2) If $W$ has the 
form $W:x \mapsto e^{a x^2 + b x +c}$, with $a,b,c\in\R$ such that $a<0$ and the two numbers $\sqrt{1-a}$ and $b$ are algebraically independent, then system  \r{1Dosc} is approximately controllable, provided that $U$ has nonempty interior.
\end{prop}
\begin{proof}
 Since each function $\phi_k$ has the same parity as the integer $k$, then $\phi_k \phi_j$ has the same parity as the integer $j+k$. If $W$ is even, then 
 \r{def-bjk} shows that 
 for every $(j,k)$ such that $j+k$ is odd, $b_{jk}=0$. Applying Remark~\ref{connectedness}, one sees that the spaces spanned by the sets $\{\phi_k\mid k \mbox{ even}\}$ and $\{\phi_k\mid k \mbox{ odd}\}$ are invariant by the dynamics of system \r{1Dosc}. In particular, there is no way to steer system \r{1Dosc} from $\phi_1$ to a point $\eps$-close to $\phi_2$ if $\eps$ is smaller than $\sqrt{2}$. This proves (1).

In order to prove (2) let us apply Corollary~\ref{complete} (with 
$U$ playing the role of $\tilde U$ and $\R$ the role of $U$).
Let $W$ have the special form $W:x \mapsto e^{a x^2 + b x +c}$. Up to a multiplication of $W$ by the strictly positive real number 
$e^{\frac{b^2}{4(a-1)} - {c}}$,  we may assume without loss of generality that 
\be\label{nrmy}
c=\frac{b^2}{4(a-1)}.
\ee

Using the specific expression \r{e-funx} of $\phi_k$ in the definition
of $b_{jk}$ we can write 
$$b_{jk}= (-1)^j \sigma_j \sigma_k \int_{\R} e^{a x^2 + b x +c} H_k(x) \frac{ d^j}{dx^j} e^{-x^2} dx,$$ 
with $\sigma_l=1/\sqrt{l!2^l\sqrt{\pi}}$, $l=k,j$. Notice that  
$H_k$ is a polynomial with rational coefficients and of degree $k$, whose leading coefficient is equal to $2^k$.
Integrating by parts $j$ times, we get 
$$b_{jk}=  \sigma_j \sigma_k \int_{\R} e^{(a-1) x^2 + b x +c}P_{j,k}(x) dx$$
where $P_{j,k}$ is a polynomial of degree $j+k$.
Define $(g_m^{j,k})_{m=0}^{j+k}$ through
$$P_{j,k}(x)=\sum_{m=0}^{j+k}g_m^{j,k} x^m.$$ 
Each $g_m^{j,k}$ can be seen as the evaluation 
at $b$ of a polynomial $G_m^{j,k}$ with coefficients in $\Q[a]$ whose degree is less than or equal to 
$j$. If $m\in\{k,k+1,\dots,k+j\}$ then $G_m^{j,k}$ has exactly degree $j+k-m$ and the coefficient  
corresponding to the monomial of order $j+k-m$ is $2^{m}a^{m-k}$.

The renormalization of $c$ performed above is such that
$ (a-1) x^2 + bx +c  = (a-1) \left ( x+ \frac{b}{2\sqrt{1-a}} \right )^2$. Hence, the change of variables $y=\sqrt{1-a} \left ( x + \frac{b}{2\sqrt{1-a}} \right )$ yields
$$b_{jk}= \frac{\sigma_j \sigma_k}{\sqrt{1-a}}  \int_{\R} e^{-y^2}\,P_{j,k}\!\lp \frac{y}{\sqrt{1-a}}-\frac{b}{ 2(a-1)}\rp dy.$$ 

Due to the remarks made above on the coefficients of $P_{j,k}$, we have 
\br
P_{j,k}\!\lp \frac{y}{\sqrt{1-a}}-\frac{b}{ 2(a-1)}\rp&=&\sum_{m=k}^{j+k}2^m a^{m-k}b^{j+k-m}\lp \frac{-b}{2 (a-1)} \rp^m +Q_{j,k}(b,y)\nonumber\\
&=&\frac{(-1)^k}{(a-1)^k}\frac{1-\lp \frac{a}{a-1}\rp^{j+1}}{1-\frac{a}{a-1}} b^{j+k}+Q_{j,k}(b,y)\label{not-here}
\er
where $Q_{j,k}$ is a polynomial with coefficients in $\Q(\sqrt{1-a})$ ($\supset \Q[a]$) and of degree smaller than $j+k$ in its first variable. Notice that the coefficient multiplying $b^{j+k}$ in \r{not-here} is different from zero.

For every $m\geq 0$ the integral $\int_{\R} e^{-y^2} y^m dy$
is equal to zero if $m$ is odd and to
$ \Gamma \left ( \frac{m+1}{2} \right )=\frac{m!}{2^m (\frac{m}{2} )!} \sqrt{\pi}$ if $m$ is even, where $\Gamma$ is the Euler gamma function.

Therefore, if $j+k$ is even,
$$b_{jk}= \frac{\sigma_j \sigma_k\sqrt\pi}{\sqrt{1-a}}  
S_{j,k}(b)$$ 
where $S_{j,k}$ is a polynomial with coefficients in $\Q(\sqrt{1-a})$ of degree exactly $j+k$.

Since $b$ is transcendental over $\Q(\sqrt{1-a})$ then $b_{jk}\ne 0$ as soon as $j$ and $k$ have the same parity. Moreover, the elements of the sequence $(\Lambda'_{k}(0))_{k\geq 0}=(b_{kk})_{k\geq 0}$ are  $\Q$-linearly independent. 

To conclude the proof let us check that each matrix $(b_{jk})_{j,k=0}^n$ is connected.
Fix $j,k\in\{0,\dots,n\}$. We should prove the existence of
a sequence $r_1,\dots,r_l\in\{0,\dots,n\}$ such that $b_{j r_1}b_{r_1 r_2}\cdots b_{r_{l-1}r_l}b_{r_l k}\ne 0$. If $j$ and $k$ have the same parity then we are done since $b_{jk}\ne 0$.  Otherwise, a simple computation and the normalization \r{nrmy} show that 
$$b_{01}=\frac{b}{\sqrt{2} (1-a)^{3/2}}\ne 0 $$
and we can conclude by taking $\{r_1,r_2\}=\{0,1\}$. 
\end{proof}

\subsection{3D potential well}
Consider the Schr\"{o}dinger equation
\be\label{3Dpw}
i \frac{\partial \psi}{\partial t}(t,x) = - \Delta \psi(t,x) + 
u(t) W(x) \psi(t,x),
\ee
where the wave function $\psi$ depends on the time  $t$ and on three 
space variables $x_1,x_2,x_3$ with $(x_1,x_2,x_3)\in (0,l_1)\times (0,l_2)\times (0,l_3)=\Omega$ and satisfies the Dirichlet boundary condition $\psi|_{\partial \Omega}=0$. 
Notice that the potential corresponding to the uncontrolled 
\Sch\ operator is $V(x)=0$.
For every $W$ measurable bounded,
solutions in the sense \r{solu} coincide with mild or strong solutions, depending on the regularity of the initial condition.

The spectrum of the \Sch\ operator is 
$$\left \{\lambda_{k_1,k_2,k_3}=\pi^2\lp \frac{k_1^2}{l_1^2}+\frac{k_2^2}{l_2^2}+\frac{k_3^2}{l_3^2}\rp \mid k_1,k_2,k_3\geq 1 \right \}.$$
For the sake of simplicity, assume that $(l_1 l_2)^2$, $(l_1 l_3)^2$, and $(l_2 l_3)^2$ are $\Q$-linearly independent, so that all the eigenvalues are simple and the perturbation result appearing in Theorem~\ref{the-perturbation}
can be applied. (The case of multiple eigenvalues can be treated similarly, applying 
a refined perturbation argument as the one used in \cite{Albert}.)

The normalized eigenfunction corresponding to 
$\lambda_{k_1,k_2,k_3}$ is given, up to sign, by
$$\phi_{k_1,k_2,k_3}(x_1,x_2,x_3)=\frac{2^{\frac32}}{\sqrt{l_1 l_2 l_3}}\sin\lp \frac{k_1 x_1 \pi}{l_1}\rp\sin\lp \frac{k_2 x_2 \pi}{l_2}\rp\sin\lp \frac{k_3x_3 \pi}{l_3}\rp.$$

\begin{prop}\label{p-box}
Let $(l_1 l_2)^2$, $(l_1 l_3)^2$, and $(l_2 l_3)^2$ be $\Q$-linearly independent and define $W(x_1,x_2,x_3)=e^{\al_1 x_1+\al_2 x_2+\al_3 x_3}$  with $\al_1,\al_2,\al_3\in\R$. Assume that $\al_1,\al_2,\al_3$ are nonzero and that 
$(\pi/\al_1 l_1)^2,$ $(\pi/\al_2 l_2)^2,(\pi/\al_3 l_3)^2$ are algebraically independent. Then 
 the control system \r{3Dpw} is approximately controllable.    
\end{prop}
Before starting the proof of Proposition~\ref{p-box} let us show the following technical result.
\begin{lem}\label{vand}
Let $\beta$ be a real number transcendental over a field $\F$ with $\Q\subset\F\subset\R$. Then the elements of the family $\lp \frac1{1+q \beta}\rp_{q\in\Q}$ are $\F$-linearly independent. 
\end{lem}
\begin{proof}
Fix $N\in\N$ and $N$ distinct numbers $q_1,\dots,q_N\in \Q\setminus \{0\}$. Assume that for some $f_1,\ldots f_N$ in $\F$
\be\label{zum}
\sum_{k=1}^N f_k \frac1{1+q_k \beta}=0.
\ee
We have to prove that $f_1=f_2=\cdots=f_N=0$.
Multiplying \r{zum} by $\Pi_{k=1}^N (1+q_k\beta)$ we get
\be\label{zzum}
\sum_{k=1}^N f_k\lp \sum_{r=0}^{N-1} s_{k,r} \beta^r\rp=0
\ee
where
$s_{k,0}=1$ and, for $r\geq1$,
$$s_{k,r}=\sum_{{\scriptsize \ba{c}1\leq j_1<j_2<\cdots< j_r\leq N\\
 j_1,\dots,j_r\ne k\ea}}q_{j_1} q_{j_2} \cdots q_{j_r}.
$$
By hypothesis, all coefficients of the left-hand side of \r{zzum}, seen as a polynomial in $\beta$, are equal to zero. Hence, $(f_1,\dots,f_N) S_N=(0,\dots,0)$ where
$$S_N=\lp\ba{ccc}
s_{1,0}&\cdots&s_{1,N-1}\\
\vdots&&\vdots\\
s_{N,0}&\cdots&s_{N,N-1}
\ea\rp.$$
A computation shows that 
$\mathrm{det}(S_N)= \Pi_{1\leq j <k\leq N} (q_k-q_j)$. 
Hence,  $S_N$ is invertible and therefore $f_1=f_2=\cdots=f_N=0$. 
\end{proof}

\noindent {\it Proof of Proposition~\ref{p-box}}.
Theorem~\ref{the-perturbation} and Fubini's theorem
imply that the eigenvalues $\Lambda_{k_1,k_2,k_3}(\mu)$ 
of $-\Delta+\mu W$ on $\Omega$ 
for the Dirichlet boundary value problem 
satisfy
\brs\label{deri3D}
\Lambda'_{k_1,k_2,k_3}(0)&=&
\frac{64 (e^{\al_1 l_1}-1)(e^{\al_2 l_2}-1)(e^{\al_3 l_3}-1) {k_1}^2 {k_2}^2 {k_3}^2 \pi^6}
{{\al_1 l_1}{\al_2 l_2}{\al_3 l_3} 
(4    \pi ^2 k_1^2+\al_1^2 l_1^2)(4    \pi ^2 k_2^2+\al_2^2 l_2^2)(4    \pi ^2 k_3^2+\al_3^2 l_3^2)}\nonumber\\
&=&C k_1^2 k_2^2 k_3^2
\frac{1}{\lp \frac{4    \pi^2 }{\al_1^2 l_1^2} k_1^2+1\rp\lp \frac{4    \pi^2 }{\al_2^2 l_2^2} k_2^2+1\rp\lp \frac{4    \pi^2 }{\al_3^2 l_3^2} k_3^2+1\rp},\label{machin}
\ers
where
$$C=\frac{64 (e^{\al_1 l_1}-1)(e^{\al_2 l_2}-1)(e^{\al_3 l_3}-1)  \pi^6}
{({\al_1 l_1}{\al_2 l_2}{\al_3 l_3})^3 
}. $$
Let $\beta_j=4    \pi^2 /(\al_j^2 l_j^2)$, $j=1,2,3$. The $\Q$-linear independence of 
the elements of 
$({\Lambda'_{k_1,k_2,k_3}}(0))_{k_1,k_2,k_3\in\N}$ 
is obtained from the expression above 
thanks to three nested applications of Lemma~\ref{vand} 
with $\F=\Q(\beta_1,\beta_2)$ and $\beta=\beta_3$, $\F=\Q(\beta_1)$ and $\beta=\beta_2$, and $\F=\Q$ and $\beta=\beta_1$.
In order to complete the proof, let us check that every matrix $B^{(n)}$ is connected. 
(The conclusion then follows from Corollary~\ref{complete}.) 
A straightforward computation shows that for every
triples of positive integers $(k_1,k_2,k_3)$ and $(h_1,h_2,h_3)$ the integral
$$\int_\Omega e^{\al_1 x_1+\al_2 x_2+\al_3 x_3}\phi_{k_1,k_2,k_3}(x_1,x_2,x_3)\phi_{h_1,h_2,h_3}(x_1,x_2,x_3)dx_1 dx_2 dx_3$$
is different from zero, i.e., every element of $B^{(n)}$ is nonzero.
\hfill$\Box$\\

{\bf Acknowledgments.} The authors are grateful to Andrei Agrachev for inspiring this work and to George Weiss,  Marius Tucsnak,  Riccardo Adami,  Anne de Roton, Tak\'eo Takahashi for helpful discussions.

{\small
\bibliographystyle{abbrv}
\bibliography{biblio}

\begin{thebibliography}{10}

\bibitem{Boscain_Adami}
R.~Adami and U.~Boscain.
\newblock Controllability of the {S}chr\"odinger equation via intersection of
  eigenvalues.
\newblock In {\em Proceedings of the 44th IEEE Conference on Decision and
  Control, December 12-15}, pages 1080--1085, 2005.

\bibitem{agrachev_chambrion}
A.~Agrachev and T.~Chambrion.
\newblock An estimation of the controllability time for single-input systems on
  compact {L}ie groups.
\newblock {\em ESAIM Control Optim. Calc. Var.}, 12(3):409--441, 2006.

\bibitem{shirikyan}
A.~Agrachev, S.~Kuksin, A.~Sarychev, and A.~Shirikyan.
\newblock On finite-dimensional projections of distributions for solutions of
  randomly forced 2{D} {N}avier-{S}tokes equations.
\newblock {\em Ann. Inst. H. Poincar\'e Probab. Statist.}, 43(4):399--415,
  2007.

\bibitem{book2}
A.~A. Agrachev and Y.~L. Sachkov.
\newblock {\em Control theory from the geometric viewpoint}, volume~87 of {\em
  Encyclopaedia of Mathematical Sciences}.
\newblock Springer-Verlag, Berlin, 2004.
\newblock Control Theory and Optimization, II.

\bibitem{Navier-Stokes}
A.~A. Agrachev and A.~V. Sarychev.
\newblock Controllability of 2{D} {E}uler and {N}avier-{S}tokes equations by
  degenerate forcing.
\newblock {\em Comm. Math. Phys.}, 265(3):673--697, 2006.

\bibitem{Albert}
J.~H. Albert.
\newblock Genericity of simple eigenvalues for elliptic {PDE}'s.
\newblock {\em Proc. Amer. Math. Soc.}, 48:413--418, 1975.

\bibitem{albertini}
F.~Albertini and D.~D'Alessandro.
\newblock Notions of controllability for bilinear multilevel quantum systems.
\newblock {\em IEEE Trans. Automat. Control}, 48(8):1399--1403, 2003.

\bibitem{altafini1}
C.~Altafini.
\newblock Controllability of quantum mechanical systems by root space
  decomposition of {${su}(N)$}.
\newblock {\em J. Math. Phys.}, 43(5):2051--2062, 2002.

\bibitem{altafini}
C.~Altafini.
\newblock Controllability properties for finite dimensional quantum {M}arkovian
  master equations.
\newblock {\em J. Math. Phys.}, 44(6):2357--2372, 2003.

\bibitem{bms}
J.~M. Ball, J.~E. Marsden, and M.~Slemrod.
\newblock Controllability for distributed bilinear systems.
\newblock {\em SIAM J. Control Optim.}, 20(4):575--597, 1982.

\bibitem{bkp}
L.~Baudouin, O.~Kavian, and J.-P. Puel.
\newblock Regularity for a {S}chr\"odinger equation with singular potentials
  and application to bilinear optimal control.
\newblock {\em J. Differential Equations}, 216(1):188--222, 2005.

\bibitem{Beauchard1}
K.~Beauchard.
\newblock Local controllability of a 1-{D} {S}chr\"odinger equation.
\newblock {\em J. Math. Pures Appl. (9)}, 84(7):851--956, 2005.

\bibitem{beauchard-coron}
K.~Beauchard and J.-M. Coron.
\newblock Controllability of a quantum particle in a moving potential well.
\newblock {\em J. Funct. Anal.}, 232(2):328--389, 2006.

\bibitem{borzi}
A.~Borz{\`{\i}} and E.~Decker.
\newblock Analysis of a leap-frog pseudospectral scheme for the {S}chr\"odinger
  equation.
\newblock {\em J. Comput. Appl. Math.}, 193(1):65--88, 2006.

\bibitem{boscain4}
U.~Boscain, T.~Chambrion, and G.~Charlot.
\newblock Nonisotropic 3-level quantum systems: complete solutions for minimum
  time and minimum energy.
\newblock {\em Discrete Contin. Dyn. Syst. Ser. B}, 5(4):957--990 (electronic),
  2005.

\bibitem{boscain3}
U.~Boscain and G.~Charlot.
\newblock Resonance of minimizers for {$n$}-level quantum systems with an
  arbitrary cost.
\newblock {\em ESAIM Control Optim. Calc. Var.}, 10(4):593--614 (electronic),
  2004.

\bibitem{q5}
U.~Boscain and P.~Mason.
\newblock Time minimal trajectories for a spin {$1/2$} particle in a magnetic
  field.
\newblock {\em J. Math. Phys.}, 47(6):062101, 29, 2006.

\bibitem{coron-libro}
J.-M. Coron.
\newblock {\em Control and nonlinearity}, volume 136 of {\em Mathematical
  Surveys and Monographs}.
\newblock American Mathematical Society, Providence, RI, 2007.

\bibitem{dalessandro-book}
D.~D'Alessandro.
\newblock {\em {Introduction to quantum control and dynamics.}}
\newblock {Applied Mathematics and Nonlinear Science Series. Boca Raton, FL:
  Chapman, Hall/CRC.}, 2008.

\bibitem{davies}
E.~B. Davies.
\newblock {\em Spectral theory and differential operators}, volume~42 of {\em
  Cambridge Studies in Advanced Mathematics}.
\newblock Cambridge University Press, Cambridge, 1995.

\bibitem{henrot}
A.~Henrot.
\newblock {\em Extremum problems for eigenvalues of elliptic operators}.
\newblock Frontiers in Mathematics. Birkh\"auser Verlag, Basel, 2006.

\bibitem{glaser}
P.~H\"ubler, J.~Bargon, and S.~J. Glaser.
\newblock Nuclear magnetic resonance quantum computing exploiting the pure spin
  state of para hydrogen.
\newblock {\em J. Chem. Phys.}, 113(6):2056--2059, 2000.

\bibitem{ito}
K.~Ito and K.~Kunisch.
\newblock Optimal bilinear control of an abstract {S}chr\"odinger equation.
\newblock {\em SIAM J. Control Optim.}, 46(1):274--287 (electronic), 2007.

\bibitem{jur}
V.~Jurdjevic and H.~J. Sussmann.
\newblock Control systems on {L}ie groups.
\newblock {\em J. Differential Equations}, 12:313--329, 1972.

\bibitem{katino}
T.~Kato.
\newblock {\em Perturbation theory for linear operators}.
\newblock Die Grundlehren der mathematischen Wissenschaften, Band 132.
  Springer-Verlag New York, Inc., New York, 1966.

\bibitem{brockett}
N.~Khaneja, S.~J. Glaser, and R.~Brockett.
\newblock Sub-{R}iemannian geometry and time optimal control of three spin
  systems: quantum gates and coherence transfer.
\newblock {\em Phys. Rev. A (3)}, 65(3, part A):032301, 11, 2002.

\bibitem{mirra-solo}
M.~Mirrahimi.
\newblock Lyapunov control of a particle in a finite quantum potential well.
\newblock In {\em Proceedings of the 45th IEEE Conference on Decision and
  Control}, December 13-15, 2006.

\bibitem{mira_rouch}
M.~Mirrahimi and P.~Rouchon.
\newblock Controllability of quantum harmonic oscillators.
\newblock {\em IEEE Trans. Automat. Control}, 49(5):745--747, 2004.

\bibitem{peirs}
A.~Peirce, M.~Dahleh, and H.~Rabitz.
\newblock Optimal control of quantum mechanical systems: Existence, numerical
  approximations, and applications.
\newblock {\em Phys. Rev. A}, 37:4950--4964, 1988.

\bibitem{pierfi}
V.~Pierfelice.
\newblock Strichartz estimates for the {S}chr\"odinger and heat equations
  perturbed with singular and time dependent potentials.
\newblock {\em Asymptot. Anal.}, 47(1-2):1--18, 2006.

\bibitem{science}
H.~Rabitz, H.~de~Vivie-Riedle, R.~Motzkus, and K.~Kompa.
\newblock Wither the future of controlling quantum phenomena?
\newblock {\em SCIENCE}, 288:824--828, 2000.

\bibitem{reed_simon}
M.~Reed and B.~Simon.
\newblock {\em Methods of modern mathematical physics. {IV}. {A}nalysis of
  operators}.
\newblock Academic Press [Harcourt Brace Jovanovich Publishers], New York,
  1978.

\bibitem{Rellich}
F.~Rellich.
\newblock {\em Perturbation theory of eigenvalue problems}.
\newblock Assisted by J. Berkowitz. With a preface by Jacob T. Schwartz. Gordon
  and Breach Science Publishers, New York, 1969.

\bibitem{rodni}
I.~Rodnianski and W.~Schlag.
\newblock Time decay for solutions of {S}chr\"odinger equations with rough and
  time-dependent potentials.
\newblock {\em Invent. Math.}, 155(3):451--513, 2004.

\bibitem{rodrigues}
S.~S. Rodrigues.
\newblock Navier-{S}tokes equation on the rectangle controllability by means of
  low mode forcing.
\newblock {\em J. Dyn. Control Syst.}, 12(4):517--562, 2006.

\bibitem{roucho2}
P.~Rouchon.
\newblock Control of a quantum particle in a moving potential well.
\newblock In {\em Lagrangian and Hamiltonian methods for nonlinear control
  2003}, pages 287--290. IFAC, Laxenburg, 2003.

\bibitem{such}
Y.~L. Sachkov.
\newblock Controllability of invariant systems on {L}ie groups and homogeneous
  spaces.
\newblock {\em J. Math. Sci. (New York)}, 100(4):2355--2427, 2000.
\newblock Dynamical systems, 8.

\bibitem{shapiro}
M.~{Shapiro} and P.~{Brumer}.
\newblock {\em {Principles of the Quantum Control of Molecular Processes}}.
\newblock Principles of the Quantum Control of Molecular Processes,
  pp.~250.~Wiley-VCH, Feb. 2003.

\bibitem{TeTuc}
G.~Tenenbaum, M.~Tucsnak, K.~Ramdani, and T.~Takahashi.
\newblock A spectral approach for the exact observability of infinite
  dimensional systems with skew-adjoint generator.
\newblock {\em to appear in Journal of Functional Analysis}, 2007.

\bibitem{turinici}
G.~Turinici.
\newblock On the controllability of bilinear quantum systems.
\newblock In M.~Defranceschi and C.~Le~Bris, editors, {\em Mathematical models
  and methods for ab initio Quantum Chemistry}, volume~74 of {\em Lecture Notes
  in Chemistry}. Springer, 2000.

\bibitem{Review_zuazua}
E.~Zuazua.
\newblock Remarks on the controllability of the {S}chr\"odinger equation.
\newblock In {\em Quantum control: mathematical and numerical challenges},
  volume~33 of {\em CRM Proc. Lecture Notes}, pages 193--211. Amer. Math. Soc.,
  Providence, RI, 2003.

\end{thebibliography}
}
	 
\end{document}